\documentclass[11pt,reqno]{amsart}
 \usepackage[width=150mm,height=220mm,
 centering,
 ]{geometry}
\usepackage{amsmath,amsthm,amsfonts,amssymb,amscd}
\usepackage{thmtools}
\usepackage{xcolor}
\usepackage{tikz}
\usepackage[noadjust]{cite}
\allowdisplaybreaks
\numberwithin{equation}{section}

\newtheorem{definition}{Definition}[section]
\newtheorem{theorem}[definition]{Theorem}
\newtheorem{proposition}[definition]{Proposition}
\newtheorem{lemma}[definition]{Lemma}
\newtheorem{corollary}[definition]{Corollary}

\newtheorem{remark}[definition]{Remark}

\def\rmd{\mathrm d}
\def\bbR{\mathbb R}
\def\bbZ{\mathbb Z}

\begin{document}
\title{Sampling Density for Gabor Phase Retrieval}

\author[T. Chen]{Ting Chen}
\address{School of Mathematical Sciences and LPMC,
Nankai University,
Tianjin,
China}
\email{t.chen@nankai.edu.cn}
	
\author{Hanwen Lu}
\address{School of Mathematical Sciences and LPMC,
Nankai University,
Tianjin,
China}
\email{2013537@mail.nankai.edu.cn}

\author[W. Sun]{Wenchang Sun}
\address{School of Mathematical Sciences and LPMC,
Nankai University,
Tianjin,
China}
\email{sunwch@nankai.edu.cn}

\date{}


\thanks{Corresponding author: Wenchang Sun.}
\thanks{This work was partially supported by the
National Natural Science Foundation of China (12271267 and 12571104 ) and the Fundamental Research Funds for the Central Universities.}

\begin{abstract}
We study the sampling density
for Gabor phase
retrieval on the full space $L^2(\mathbb R)$.
We classify the power lattices $a\mathbb Z^\nu\times b\mathbb Z^\nu$
that determine every function up to a global phase.
For $0<\nu<1$ with $\nu\ne1/2$, uniqueness holds for every
$a,b>0$, whereas it fails for every pair when $\nu\ge1$.
At the square-root scale, uniqueness holds whenever $a^2/b^2$
is irrational. If $a^2=\kappa p$ and $b^2=\kappa q$ with coprime
positive integers $p,q$, it holds precisely when $\kappa\le1$.
The critical case follows from a sharp quadratic growth bound,
classification of its equality case, and Fock integrability.
For $1/2<\nu<1$, we prove a stronger
modulus-uniqueness statement for all entire functions of finite
quadratic type using only first-quadrant samples.
We also give uniqueness and nonuniqueness conditions for samples
on two intersecting or three parallel lines, relating the sampling
scale to the geometry of the lines.
The lattice classification extends coordinatewise to
Gabor measurements on $L^2(\mathbb R^d)$.
\end{abstract}

\maketitle

\section{Introduction and main results}

We study the recovery of a signal from the magnitude of its
Gabor transform on a prescribed discrete set. Our main result classifies
the rectangular power lattices $a\mathbb Z^\nu\times b\mathbb Z^\nu$ on
the full space $L^2(\mathbb R)$. Except at $\nu=1/2$, the exponent
alone determines uniqueness. At $\nu=1/2$, lattices with the same sampling
growth can behave differently, and the distinction is arithmetic.

Given a window function $g\in L^2(\bbR)$,
we define the short time  Fourier transform of a
signal $f\in L^2(\bbR)$ by
\[
  V_gf(t,\omega)
     = \int_{\bbR} f(x) \bar g(x-t) e^{-2\pi i x \omega}
     \rmd x.
\]
Whenever the window function is the Gaussian
\[
  \varphi(x):=e^{-\pi x^2},
\]
the short time  Fourier transform is also
known as the Gabor transform.

We say that $\Lambda\subset\mathbb R^2$ does Gabor phase retrieval
if, for every $f_1,f_2\in L^2(\mathbb R)$,
\[
 |V_\varphi f_1(t,\omega)|=|V_\varphi f_2(t,\omega)|,
 \quad \forall(t,\omega)\in\Lambda
\]
implies $f_1=cf_2$ almost everywhere for a constant $|c|=1$.
Unless otherwise stated, Gabor phase retrieval refers to the full complex space
$L^2(\mathbb R)$. Through the Bargmann transform, this problem is
equivalent to determining a Fock function from its sampled modulus;
see \cite[Chapter~3]{Groechenig2001}.

The distinction between sampling with and without phase is already
apparent for a single Gaussian window. On a sufficiently dense lattice,
the complex Gabor coefficients give stable reconstruction
\cite{Christensen2016,Daubechies1992,Groechenig2001}. Their magnitudes,
however, do not give uniqueness on any lattice
\cite{AlaifariWellershoff2022}. The obstruction persists for an
arbitrary single window \cite{GrohsLiehr2023a}, whereas four suitably
chosen windows can give uniqueness on a sufficiently dense lattice
\cite{GrohsLiehrRathmair2025b}. Here the window and the signal space
are fixed; it is the placement of the samples that varies.

In \cite{GrohsLiehr2025},
Grohs and Liehr
studied Gabor phase retrieval
 for $d$-dimensional signals.
Here we cite a $1$-dimensional version.
For any $ \nu>0$, denote
\[
  \bbZ^{\nu}:= \{\pm n^{\nu}:\, n\ge 0, n\in\bbZ\}.
\]
Grohs and Liehr showed that for a large class of window functions,
the square root lattice
\[
  \Lambda:= a\bbZ^{1/2} \times b\bbZ^{1/2}
\]
does phase retrieval when $a$ and $b$ are small enough.
As a consequence, they obtained that if
\[
    0<a<\frac{1}{\sqrt{2\pi e}}
    \quad \mathrm{and}\quad
    0<b<\frac{1}{\sqrt{2\pi e}},
\]
then $\Lambda$ does Gabor phase retrieval.

Our first theorem
classifies these lattices for all $a,b>0$. In particular, uniqueness holds as soon as
one scale is at most one, with no restriction on the other. Even when
both exceed one, uniqueness may still hold; the criterion then depends
on the relation between the squared scales.

\begin{theorem} \label{thm:t1}
Let $ \Lambda=a\mathbb{Z}^{1/2}\times b\mathbb{Z}^{1/2}$, where $a,b>0$.
If $a^2/b^2\notin\mathbb Q$, then $\Lambda$ does Gabor phase
retrieval on $L^2(\mathbb R)$.

If $a^2=\kappa p$ and $b^2=\kappa q$ for coprime positive integers
$p,q$, then $\Lambda$ does Gabor phase retrieval if and only if
$\kappa\le1$.
\end{theorem}

The following is an immediate consequence.
\begin{corollary} \label{coro1.2}
Let $a, b>0$.
\begin{enumerate}
\item If $\min\{a,b\}\le 1$, then $a\bbZ^{1/2}\times b\bbZ^{1/2}$
   does Gabor phase retrieval.
\item
If $a=b$, then $a\bbZ^{1/2}\times a\bbZ^{1/2}$
   does Gabor phase retrieval if and only if $0<a\le 1$.
\end{enumerate}
\end{corollary}

The product $ab$ does not determine the answer. For example,
$(a,b)=(\sqrt2,\sqrt2)$ and $(a,b)=(2^{1/3},2^{2/3})$ give the
same leading term in the number of samples in $[-R,R]^2$: in both
cases it is $4R^4/(ab)^2=R^4$. The first lattice fails phase retrieval,
since $\kappa=2$, whereas the second gives uniqueness, since
$a^2/b^2$ is irrational.

Does this dependence on arithmetic persist for other powers? The next
theorem shows that it is confined to $\nu=1/2$.

For $a,b,\nu>0$, consider the power lattice
\[
 \Lambda:=a\bbZ^\nu\times b\bbZ^\nu.
\]
Write $\bbR_+=(0,\infty)$. For $\nu>0$, define
\[
 I_\nu=
 \begin{cases}
  \emptyset,&\nu\ge1,\\[2pt]
  \bbR_+\times\bbR_+,
     &0<\nu<1\mbox{ and }\nu\ne1/2,\\[2pt]
  \begin{aligned}
   &\{(a,b)\in\bbR_+^2:\ a^2/b^2\notin\mathbb Q\mbox{ or}\\
   &\quad(a^2,b^2)=(\kappa p,\kappa q)\mbox{ for some}\\
   &\quad0<\kappa\le1\mbox{ and coprime positive integers }p,q\},
  \end{aligned}
     &\nu=1/2.
 \end{cases}
\]

\begin{theorem} \label{thm:power}
For  $\nu>0$,
the power lattice
$a\bbZ^{\nu} \times b\bbZ^{\nu}$
does Gabor phase retrieval if and only
if $(a,b)\in I_{\nu}$.
\end{theorem}

There are two different roles for the exponent in this classification.
The number of points of the full power lattice in a disk of radius $R$
centered at the origin grows on the order of $R^{2/\nu}$.
Lemma~\ref{Th4} below shows that phase retrieval requires a counting-growth
exponent of at least two, which rules out $\nu>1$. At $\nu=1$ the obstruction is
that of the ordinary lattice. On a single row, by contrast, the count
grows only as $R^{1/\nu}$. For $0<\nu<1/2$ the samples on a row already
exceed the quadratic zero-counting bound for the entire functions used
to compare two signals.
For $1/2<\nu<1$, the count on a single row no longer gives this
conclusion. However, using rows and columns together,
Lemma~\ref{lem:power:phase}  below shows a stronger
conclusion than Gabor phase retrieval:  for any $a,b>0$ the modulus on
the first-quadrant set
\[
  \{am^\nu+ibn^\nu:m,n\geq1\}
\]
determines every entire function of finite quadratic type up to a global
phase.

The preceding theorems fix the sampling locations in advance. A different
question is whether one can construct phase-retrieval sets of low density. By perturbing Liouville
sets in Fock space, Grohs, Liehr and Rathmair
\cite{GrohsLiehrRathmair2025a} constructed phase-retrieval sets of finite
density. De Dios Pont, Greilhuber, Liehr and Taylor
\cite{DeDiosPontGreilhuberLiehrTaylor2026} constructed uniformly discrete
sets for Gaussian and polynomial-times-Gaussian windows in every
dimension. These constructions do not give a criterion for a prescribed power lattice.

We also consider samples at $a\mathbb Z^\nu$ on each of two intersecting
or three parallel lines. For complete lines the geometric criteria are
known. Jaming \cite[Theorem~3.3]{Jaming2014} showed that the modulus of
an entire function of finite order on two lines determines it up to a
global phase if their angle is an irrational multiple of $\pi$.
 Jaming   also
observed that suitable uniqueness sets on the lines suffice. For three distinct
parallel lines, Wellershoff \cite{Wellershoff2024a} proved uniqueness
for arbitrary entire functions when the distances from one line to
the other two have irrational ratio. More generally, Liehr
\cite{Liehr2024} characterized uniqueness for a family of parallel
lines by the condition that its transverse coordinates do not lie in
an arithmetic progression.

Theorems~\ref{thm:t2a} and~\ref{thm:t2b} give explicit sampling
conditions for these geometries. Under the respective irrationality
conditions, uniqueness holds for all $a>0$ if $\nu<1/2$, and for
$a<1$ if $\nu=1/2$. For $\nu>1/2$ it always fails: the count on the
entire finite family of lines now grows as $R^{1/\nu}$, too slowly for
uniqueness. At $\nu=1/2$ we give further nonuniqueness conditions
that depend on the geometry.

These results concern all signals in $L^2(\mathbb R)$. Different sampling sets can suffice if the signal class is restricted,
as in the case of bandlimited signals \cite{AlaifariWellershoff2021a,Wellershoff2023,
ZhangGuoLiuLi2023}, compactly supported signals
\cite{GrohsLiehr2023b,Wellershoff2024b}, and signals in shift-invariant
or spline spaces \cite{ChenSun2022,ChenChengSunWang2020,ChengWuXian2025,
Grochenig2020,LiLiuZhang2021,ShenoyMulletiSeelamantula2016,Sun2021}.
Related settings include complex- and vector-valued functions
\cite{ChenChengSun2022}, backward-shift-invariant subspaces
\cite{QianTan2016}, and quaternionic signals \cite{LiYang2025,YangLi2024}.
For other forms of phaseless data, see \cite{Gao2025,GaoSunWangXu2018,HuangXu2024}
on affine measurements, \cite{LiLi2021a,LiLi2021b} on masked Fourier
measurements, and \cite{LiWeiFan2024} on interference measurements.

Stability has been studied in infinite-dimensional spaces and continuous
frames~\cite{AlaifariDaubechiesGrohsYin2019,AlaifariGrohs2017,
CahillCasazzaDaubechies2016,FreemanOikhbergPineauTaylor2024},
as well as for Gabor and related
measurements~\cite{AlaifariGrohs2021,AlaifariWellershoff2021b,Grohs2019,
Grohs2022,GrohsLiehr2024,Alaifari2024,Zhong2023}.
Further work addresses robustness to outliers~\cite{HuangLiXu2025,HuangLiXu2026},
condition numbers~\cite{XiaXuXu2025}, and instance
optimality~\cite{XiaXu2026}. For surveys of the mathematical theory
and of applications to imaging, see \cite{GrohsKoppensteinerRathmair2020}
and \cite{Shechtman2015}, respectively.

The proofs start with the Bargmann transform, which turns equality of
sampled Gabor magnitudes into equality of the moduli of two Fock
functions at the corresponding points. In Section~\ref{sec:regular},
reflection across a sampling line expresses this equality as the
vanishing of an entire function at the sampled points. For $\nu<1/2$,
zero counting extends the modulus identity to the whole line. At
$\nu=1/2$, the zero count has quadratic growth. Levin's sector theorem
then relates the density of the zeros on a ray to growth along the
boundary rays of a sector and gives the same extension for $a<1$.

Section~\ref{sec:square-root} proves the classification at the square-root
scale. If uniqueness fails, Lemma~\ref{Lm:L5} gives linearly independent
Fock functions $U,V$ such that $V\overline U$ is real at every sampling
node in the complex plane. Neighboring nodes impose integer constraints
on a logarithmic derivative in each coordinate direction. Compatibility
of the two constraints forces $a^2/b^2$ to be rational. Write
$a^2=\kappa p$ and $b^2=\kappa q$ with $p,q$ coprime positive integers.
Lemma~\ref{lem:arithmetic:type} gives a sharp lower bound on the
quadratic growth of the pair. At $\kappa=1$, the Fock growth bound forces
equality, and the equality case identifies the span of $U,V$ with that
of $e^{i\pi z^2/2}$ and $e^{-i\pi z^2/2}$. Both exponentials fail the
weighted square-integrability condition for Fock space, which excludes
such a pair. For $\kappa>1$, the exponentials
$e^{\pm i\pi z^2/(2\kappa)}$ belong to Fock space and their quotient is
real at the lattice nodes, giving nonuniqueness.

Section~\ref{sec:power} treats $1/2<\nu<1$. On a suitable grid in the
first quadrant, the phase increments along rows and columns are integer
multiples of $\pi$. Estimates for the derivatives of the phase first
show that the increments are constant in each direction on this lattice.
The variation in the spacings between consecutive nodes then forces
both increments to vanish. This rules out a linearly independent pair
and gives the first-quadrant uniqueness result in
Lemma~\ref{lem:power:phase}.

In Section~\ref{sec:multivariate},
the power-lattice classification is extended to Gabor phase
retrieval of multivariate functions.
A product sampling set of this form gives
Gabor phase retrieval precisely when the one-dimensional criterion holds
in every coordinate. The exponents may vary from one coordinate to the next.

\section{Preliminaries}\label{sec:preliminaries}

In this section, we collect some preliminaries
to be used in the proofs.

For two functions $f$ and $g$, $f\sim g$ means that
there is some $\alpha\in [0,2\pi]$ such that
$f(x) = e^{i\alpha} g(x)$ for almost all $x\in\bbR$.
For entire functions, the corresponding equality is understood on
$\mathbb C$.

$ \mathcal{F}f  $ and $\hat f$ denote the Fourier transform of a function $f$, i.e.,
\[
  \mathcal F f(\omega) = \hat f(\omega)
  := \int_{\bbR} f(x) e^{-2\pi i x\omega}\rmd x.
\]

Given a sequence $\Lambda\subset \bbR^2$,
define the associated sequence  in $\mathbb C$ by
\begin{equation*}
\Gamma^*_\Lambda=\{x+iy:\,(y,x)\in \Lambda\} .
\end{equation*}

\subsection{The Bargmann transform}
Let $ \rmd\mu(z)=e^{-\pi|z|^2}\rmd x\, \rmd y $
be the Gaussian measure. Recall that the Bargmann-Fock space $ F^2(\mathbb{C}) $
consists of all entire functions $f$ for which
\[
    \|f\|_{F^2(\mathbb{C})}
    :=\Big(\int_\mathbb{C}|f(z)|^2 \rmd\mu(z)\Big)^{1/2}<\infty.
\]

For any $ f\in L^2(\mathbb{R}) $,
the Bargmann transform of $f$ is defined by  \begin{equation}\label{Eq13}
Bf(z)=2^{1/4}
 \int_\mathbb{R}f(s)
 e^{2\pi sz-\pi s^2-\pi z^2/2}\rmd s,\quad z\in\mathbb C.
\end{equation}
The Bargmann transform is a unitary
operator from  $ L^2(\mathbb{R}) $ onto $ F^2(\mathbb{C}) $.
We see from (\ref{Eq13})
 that the Bargmann and  Gabor transforms
are related by
\[
V_\varphi f(t,-\omega)=2^{-1/4}e^{\pi it\omega- \pi |z|^2/2}Bf(z),
\]
where $ z=t+i\omega\in\mathbb{C}$, $t,\omega\in\mathbb{R} $.
For an introduction on Bargmann transforms,
we refer to \cite[Chapter 3]{Groechenig2001}.

\subsection{Counting-growth exponent of a sampling set}

Let $ \Lambda $ be a set in $ \mathbb{R}^2 $ with no accumulation points. Define the counting function $ n_\Lambda(r)$ for $ \Lambda $ by
\[
  n_\Lambda(r): =\#(\Lambda\cap\{\lambda\in\mathbb{R}^2:
 \,|\lambda|\le r\}),
\]
where $\#$ stands for the cardinality of a sequence or a set.
The counting-growth exponent of $\Lambda$ is defined by
\begin{equation*}
L_\Lambda:=\limsup\limits_{r\rightarrow\infty}\frac{\log\max\{1,n_\Lambda(r)\}}{\log r}.
\end{equation*}
For a set  $ \Gamma $ in the complex plane $ \mathbb{C} $ with no accumulation points, $ n_{\Gamma}(r) $ and $ L_{\Gamma} $
are defined similarly.

\subsection{The canonical product}

Here we introduce some results for entire functions.
Define the elementary factors by $E_0(z) = 1-z$
and
\[
E_k(z)=(1-z)\exp\{ z+\frac{z^2}{2}+\cdot\cdot\cdot+\frac{z^k}{k} \},
\quad k\ge 1.
\]

Let $ \Gamma=\{z_n:\,n\geq 1, z_n\neq0\} $
be a sequence in $\mathbb C$.
Define the convergence exponent  of $\Gamma$ by
\[
    \rho_\Gamma=\inf\Big\{\alpha>0:\,\sum_{n=1}^\infty \frac{1}{|z_n|^{\alpha}}<\infty\Big\}.
\]
If $\rho_\Gamma<\infty$, let $ \alpha $ be the minimal positive integer for which the series converges.
We call
$ p:=\alpha-1 $
the genus of $\Gamma$.

Whenever $\Gamma$ is the zero set of an entire function $f$, $z_n$ must be listed according to their multiplicities. In this case, we also call
$p$ the genus of zero set of $f$.

For the zeros of a nonzero entire function, define the convergence
exponent after omitting any zero at the origin, and use the
counting function with multiplicities. A zero at the origin changes
the counting function by only a finite amount. With these conventions,
\begin{equation}\label{eq:rho}
\rho_\Gamma = L_\Gamma .
\end{equation}
We refer to \cite[Theorem 2.5.8]{Boas1954}
for a proof.

The order of a nonzero entire function $f$ is defined by
\[
    \varrho=\limsup_{r\rightarrow\infty} \frac{\log\log\max\{e,M(r)\}} {\log r},
\]
where $M(r)$ denotes the maximum of
$ |f(z)| $ on $\{z:\, |z|\le r\}$.

The type $\tau$ of an entire function $f$
with positive order $\varrho<\infty$ is defined by \[
    \tau:=\limsup_{r\rightarrow\infty} \frac{ \log M(r)} {r^{\varrho}}.
\]

Let $\Gamma:= \{ z_n:\,n\geq1 \}  $ be a complex sequence,
$ z_n\neq0 $ and  $ \lim\limits_{n\rightarrow\infty} |z_n|={\infty}  $.
Suppose that the genus of $\Gamma$ equals $p$.
Then
\[
    P(z):=\prod_{n=1}^\infty E_p (\frac{z} {z_n})
\]
defines an entire function $P$, for which $p$ is
the genus of its zero set and   $\Gamma$
is the zero set~\cite[Theorem 15.9]{Rudin1987}.
We call $P$   a
canonical product of genus $p$.

\begin{proposition} [{\cite[Theorem 2.6.5]{Boas1954}}]\label{Prop5}
A canonical product of genus $p$
is an entire function of
order equal to the convergence exponent of its zeros.
\end{proposition}

\subsection{Indicators and Levin's sector theorem}

For $m>0$ and an entire function satisfying
$|f(z)|\le C e^{A|z|^m}$ with $C,A>0$, define
\[
 H_f^{[m]}(\theta)
 :=\limsup_{r\to\infty}\frac{\log|f(re^{i\theta})|}{r^m}.
\]
We set $H_0^{[m]}=-\infty$. The exponent in this definition is fixed;
the actual order of $f$ may be smaller than $m$. All quadratic
indicators below are taken with $m=2$.

We recall Levin's sector theorem. Let $\varrho(r)$ be a proximate order satisfying
\[
 \lim_{r\to\infty}\varrho(r)=\varrho>0,\qquad
 \lim_{r\to\infty}r\varrho'(r)\log r=0,
\]
and the indicator is taken relative to this proximate order:
\[
 h_F(\theta)=\limsup_{r\to\infty}
                 \frac{\log|F(re^{i\theta})|}{r^{\varrho(r)}}.
\]
A function $F$, holomorphic in an angle
\[
S=\{re^{i\theta}:r>0,\ \theta_1\le\theta\le\theta_2\},
\]
is said to be at most of normal type with respect to $\varrho(r)$
if there exist constants $C,R>0$ such that
\[
|F(re^{i\theta})|
\le \exp\!\bigl(Cr^{\varrho(r)}\bigr),
\qquad
r\ge R,\ \theta_1\le\theta\le\theta_2.
\]
Let $\Xi=\{a_k\}$ be a locally finite set of prescribed zeros of $F$ in $S$, and define
\[
n_\Xi(r;\beta,\gamma)=
\#\{k:0<|a_k|\le r,\ \beta<\arg a_k<\gamma\}.
\]
We say that $\Xi$ has angular density relative to $\varrho(r)$ if there exist a nondecreasing function $\Delta_\Xi$ and a countable set $E$ of exceptional angles such that
\[
\lim_{r\to\infty}
\frac{n_\Xi(r;\beta,\gamma)}{r^{\varrho(r)}}
=
\Delta_\Xi(\gamma)-\Delta_\Xi(\beta)
\qquad(\beta,\gamma\notin E).
\]
 On the complement of $E$, the function $\Delta_\Xi$ is determined up to an additive constant.  We omit the origin  from the counting function since its argument is undefined; and this omission does not affect the angular density.
Integrals with respect to $\mathrm d\Delta_\Xi$ are understood as Stieltjes integrals, so that jumps of $\Delta_\Xi$ contribute point masses. For example, if all points of $\Xi$ lie on an interior ray $\arg z=\theta_0$ and
\[
\lim_{r\to\infty}
\frac{\#\{a_k\in \Xi:|a_k|\le r\}}{r^{\varrho(r)}}=c,
\]
then one may take
\[
\Delta_\Xi(\theta)=
\begin{cases}
0,&\theta<\theta_0,\\
c,&\theta\ge\theta_0.
\end{cases}
\]
Hence
\[
\mathrm d\Delta_\Xi=c\,\delta_{\theta_0},
\qquad
\int_{\theta_1}^{\theta_2}q(\theta)\,\mathrm d\Delta_\Xi(\theta)
=
c\,q(\theta_0)
\]
for every continuous function $q$ on $[\theta_1, \theta_2]$,
where $\delta_{\theta_0}$ denotes the Dirac measure at $\theta_0$.

These conventions follow Levin \cite[Chapter~I, \S12, \S18; Chapter~II, \S1]{Levin1980}. The following theorem relates the angular density of prescribed zeros to a lower bound for the sum of the indicators on the two boundary rays of an angle of opening $\pi/\varrho$.

\begin{proposition}[{\cite[Chapter~IV, \S3, Theorem~8]{Levin1980}}]\label{prop:levin-sector}
Let $F(z)$ be holomorphic and at most of normal type with respect
to the proximate order $\varrho(r)$ in the angle
$\alpha\le\arg z\le\alpha+\pi/\varrho$, and vanish on a set
$\Xi=\{a_k\}$ in this angle, with angular density $\Delta_\Xi(\theta)$,
and let
\[
 H_\Xi(\theta)=\frac{\pi}{\sin\pi\varrho}
   \int_{\alpha}^{\alpha+\pi/\varrho}
        \cos\bigl(\varrho(|\theta-\tilde\theta|-\pi)\bigr)
        \,\rmd\Delta_\Xi(\tilde\theta)
\]
when $\varrho$ is nonintegral and
\[
 H_\Xi(\theta)=\pi\int_{\alpha}^{\alpha+\pi/\varrho}
              \sin\bigl(\varrho|\tilde\theta-\theta|\bigr)
              \,\rmd\Delta_\Xi(\tilde\theta)
\]
when $\varrho$ is integral. Then if $F(z)$ is not identically zero,
\begin{equation*}
 h_F(\alpha)+h_F\left(\alpha+\frac{\pi}{\varrho}\right)
 \ge H_\Xi(\alpha)+H_\Xi\left(\alpha+\frac{\pi}{\varrho}\right).
\end{equation*}
\end{proposition}

\subsection{Subharmonic functions}\label{subsec:subharmonic}

We write $D(z,r)=\{w\in\mathbb C:|w-z|<r\}$. First recall
the definition of a subharmonic function.
Following \cite[Definition~1.6.1]{Hormander1990}, a function
$u:\Omega\to[-\infty,\infty)$ on an open set $\Omega\subset\mathbb C$
is called \emph{subharmonic} if it satisfies the following conditions:
\begin{enumerate}
\item $u$ is upper semicontinuous, that is $\{z:\,z\in\Omega,u(z)<s\}$ is open for every real number $s$;
\item for every compact set $K\subset\Omega$ and every real-valued
function $h$ which is continuous on $K$ and harmonic in its interior,
the inequality $u\le h$ on $\partial K$ implies $u\le h$ on $K$.
\end{enumerate}
This definition allows $u\equiv-\infty$. In particular, if $F$ is a
nonzero holomorphic function, then $\log|F|$ is subharmonic, with value
$-\infty$ at the zeros of $F$; see \cite[Corollary~1.6.6]{Hormander1990}.

Subharmonic functions satisfy the submean inequality
\[
 u(z)\le\frac1{2\pi}\int_0^{2\pi}u(z+re^{it})\,\rmd t
 \qquad\bigl(\overline{D(z,r)}\subset\Omega\bigr).
\]
Integrating with respect to the radius gives the disk version
\[
 u(z)\le\frac1{\pi r^2}\int_{D(z,r)}u(\zeta)\,\rmd A(\zeta),
\]
which will also be used below; see \cite[Theorem~1.6.3]{Hormander1990}.

We use the Laplacian $\Delta=\partial_x^2+\partial_y^2$.
If a subharmonic function $u$ is not identically $-\infty$ on any
component of $\Omega$, then $u\in L^1_{\mathrm{loc}}(\Omega)$ and
its distributional Laplacian $\Delta u$ is a nonnegative Radon measure.
Conversely, every real distribution $T$ satisfying $\Delta T\ge0$
is represented by a unique subharmonic function $u$, meaning that
$T=u$ in the sense of distributions. If $T$ is represented by a
locally integrable function, its subharmonic representative agrees
with that function almost everywhere;
see \cite[Theorem~3.2.11]{Hormander1994}.

We use $\mu_u=\Delta u$ as the Riesz measure of $u$. For the above
Laplacian,
\[
 \Delta\log|z|=2\pi\delta_0,
\]
where $\delta_0$   denotes the Dirac measure
 at the origin. When discussing
pointwise values of subharmonic limits obtained in $L^1_{\mathrm{loc}}$,
we always use their unique subharmonic representatives.

We recall the following planar forms of the compactness and
convergence theorems for subharmonic functions.

\begin{proposition}[{\cite[Theorems~3.2.12--3.2.13]{Hormander1994}}]\label{prop:subharmonic:compactness}
Let $\Omega\subset\mathbb C$ be a domain and let $\{u_j\}$ be a
sequence of subharmonic functions on $\Omega$. If $\{u_j\}$ is uniformly
bounded above on every compact subset of $\Omega$, then either
$u_j\to-\infty$ uniformly on compact subsets, or a subsequence
converges in $L^1_{\mathrm{loc}}(\Omega)$ to a subharmonic function.

If $u_j\not\equiv-\infty$ and $u_j\to u$ in the sense of
distributions, where $u$ is subharmonic, then
\[
 u_j\to u,\qquad \partial_xu_j\to\partial_xu,\qquad
 \partial_yu_j\to\partial_yu
 \quad\text{in }L^1_{\mathrm{loc}}(\Omega),
\]
where the derivatives are weak derivatives.
\end{proposition}

The next formula will be used to estimate the Riesz measure of a
subharmonic function from its growth bound.

\begin{proposition}[{\cite[Theorem~4.5.1]{Pavlovic2004}}]
\label{prop:subharmonic:jensen}
Let $u$ be subharmonic in $D(0,R)$ with $u(0)>-\infty$, and let
$\mu=\Delta u$ be its Riesz measure. Then, for $0<r<R$,
\[
 \frac1{2\pi}\int_{-\pi}^{\pi}u(re^{it})\,\rmd t-u(0)
 =\frac1{2\pi}\int_{|z|<r}\log\frac r{|z|}\,\rmd\mu(z).
\]
\end{proposition}

Finally, we recall the Phragm\'en--Lindel\"of principle for subharmonic
functions, which will be used to extend an upper bound on the boundary
of a sector to its interior.

\begin{proposition}[{\cite[\S7.3, Theorem~3]{Levin1996}}]
\label{prop:subharmonic:phragmen-lindelof}
Let $D$ be an angle of opening $\pi/\lambda$, where $\lambda>0$,
and let $u$ be subharmonic in $D$. Suppose that, for some
$\rho<\lambda$,
\[
 u(z)<|z|^\rho
\]
for all $z\in D$ with sufficiently large $|z|$, and that,
for some $M\in\mathbb R$,
\[
 \limsup_{z\in D,\,\substack{z\to\zeta}}u(z)\le M,
 \qquad \zeta\in\partial D.
\]
Then $u(z)\le M$ throughout $D$.
\end{proposition}

\section{Sampling on two or three lines}\label{sec:regular}

In this section, we study Gabor phase retrieval
with samples on two intersecting lines or three
parallel lines.

We sample at $a\mathbb Z^\nu$ on each line, with $a,\nu>0$.
The resulting set $\Lambda$ has $L_\Lambda=1/\nu$. Thus the necessary
condition $L_\Lambda\ge2$ for Gabor phase retrieval on $L^2(\mathbb R)$
(Lemma~\ref{Th4}) holds with equality at $\nu=1/2$.
For $a<1$ at this scale, uniqueness holds on two
intersecting lines precisely when their angle is an irrational multiple
of $\pi$, and on three parallel lines precisely when the ratio of the
distances from the third line to the other two is irrational.

For $\theta\in\mathbb R$, write
\[
 a\mathbb Z^\nu(\sin\theta,\cos\theta)
 =\{(\lambda\sin\theta,\lambda\cos\theta):
       \lambda\in a\mathbb Z^\nu\}.
\]

For samples on two intersecting lines,
we have the following result.

\begin{theorem}\label{thm:t2a}
Suppose that
$\Lambda = \cup_{l=1}^2 \big((t_0,\omega_0)+a\bbZ^{\nu}
  (\sin\theta_l,\cos\theta_l)\big)$, where $a$ and $\nu$ are positive numbers,
$(t_0, \omega_0) \in \bbR^2$.
If  $(\theta_1 -\theta_2)/\pi
    \not\in \mathbb Q $
and either $\nu<1/2$ or $\nu=1/2$ with $a<1$,
then  $\Lambda$ does Gabor phase retrieval.

Moreover, $\Lambda$ does not do Gabor phase retrieval
if
one of the following conditions is satisfied,
\begin{enumerate}

\item $(\theta_1 -\theta_2)/\pi \in \mathbb Q$;

\item $\nu>1/2$;

\item $\nu=1/2$, $a>1 $ and  $|\cos (\theta_1-\theta_2)|>1/a^2$;

\item $\nu=1/2$, $a>1 $ and  $|\sin (\theta_1-\theta_2)|>1/a^2$;

\item $\nu=1/2$, $a>1 $ and  $|\sin(2\theta_1-2\theta_2)|>1/a^2$.
\end{enumerate}
\end{theorem}

The following theorem gives
uniqueness result for
samples on three parallel lines.

\begin{theorem}\label{thm:t2b}
Suppose that
$\Lambda = \cup_{l=1}^3 \big((t_l,\omega_l)+a\bbZ^{\nu}(\sin\theta_0,\cos\theta_0) \big)$, where $a$ and $\nu$ are positive numbers,
$\theta_0 \in [0,2\pi]$,
$(t_1,\omega_1)$,
$(t_2,\omega_2)$ and
$(t_3,\omega_3)$ are three points and
\[
  d_l:=|(\omega_l-\omega_3)\sin\theta_0
        - (t_l-t_3)\cos\theta_0|>0, \quad l=1,2.
\]
Then  $\Lambda$ does Gabor phase retrieval if
$d_1/d_2\not\in  \mathbb Q$ and either $\nu<1/2$
or  $\nu=1/2$ with $a<1$.

Moreover, $\Lambda$ does not do Gabor phase retrieval
if one of the following conditions is satisfied,
\begin{enumerate}
\item $d_1/d_2\in  \mathbb Q$;
\item $\nu>1/2$;
\item $\nu=1/2$, $a>1$,
$(t_1,\omega_1)$,
$(t_2,\omega_2)$ and
$(t_3,\omega_3)$ are in the same line whose direction is $(\cos\theta_0, -\sin\theta_0)$,
and
$d_l=a n_l^{1/2}$ for some integers $n_l$, $l=1,2$.
See Figure~\ref{fig:f1}.
\end{enumerate}
\end{theorem}

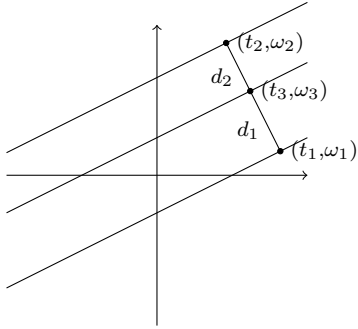
\begin{figure}[!ht]
\begin{center}
\begin{tikzpicture}[scale=1]
\draw [->] (-2,0) -- (2,0);
\draw [->] (0,-2) -- (0,2);
\draw (-2,-1.5) -- +(4,2);
\draw (-2,-0.5) -- +(4,2);
\draw (-2,0.3) -- +(4,2);

\coordinate [label=right:{$\scriptstyle (t_1,\omega_1)$}]
  (A1) at (intersection cs: first line={(1.8,0)--(0,3.6)}, second line={(-2,-1.5) -- (2,0.5)});
\filldraw (A1) circle (1pt);

\coordinate [label=right:{$\scriptstyle (t_3,\omega_3)$}]
  (A3) at (intersection cs: first line={(1.8,0)--(0,3.6)}, second line={(-2,-0.5) -- (2,1.5)});
\filldraw (A3) circle (1pt);

\coordinate [label=right:{$\scriptstyle (t_2,\omega_2)$}]
  (A2) at (intersection cs: first line={(1.8,0)--(0,3.6)}, second line={(-2,0.3) -- (2,2.3)});
\filldraw (A2) circle (1pt);

\draw (A1)--(A2);

\coordinate [label=left:{$\scriptstyle d_1$}] (D1)at (1.5,0.6);
\coordinate [label=left:{$\scriptstyle d_2$}] (D2) at (1.15,1.3);

\end{tikzpicture}
\caption{Gabor phase retrieval over
three parallel lines}\label{fig:f1}
\end{center}
\end{figure}

In the rest of this section, we prove Theorems
\ref{thm:t2a} and \ref{thm:t2b}.

\subsection{Analytic reduction and a necessary density condition}

For any $f\in L^2(\mathbb R)$, let
\begin{equation}\label{eq:Pf}
 P_f(z):=Bf(-iz)
 =2^{1/4}e^{\pi z^2/2}
   \int_{\mathbb R}f(s)e^{-\pi s^2-2\pi isz}\,\rmd s.
\end{equation}
Since rotations preserve the Fock norm, $f\mapsto P_f$ is a linear
unitary bijection from $L^2(\mathbb R)$ onto $F^2(\mathbb C)$.
Completing the square and applying Cauchy--Schwarz in
\eqref{eq:Pf} give
\begin{equation}\label{eq:Pf:growth}
 |P_f(z)|\le\|f\|_2 e^{\pi|z|^2/2},\qquad z\in\mathbb C.
\end{equation}
The relation between the Bargmann and Gabor transforms gives
\begin{equation}\label{eq:vphi:P}
 |V_\varphi f(t,\omega)|
 =2^{-1/4}e^{-\pi|z|^2/2}|P_f(z)|,
 \qquad z=\omega+it.
\end{equation}

To study phase retrieval with samples
taken from three parallel lines,
we need the following characterization
of entire functions.

\begin{proposition} [{\cite[Theorem 1]{Wellershoff2024a}}]\label{Prop3}
Suppose that $ \ell_1$, $\ell_2$, $\ell_3 \subset\mathbb{C}  $
are three parallel lines. Let $d_{k}$ be the distance between
$\ell_k$ and $\ell_3$, $k=1,2$.
If $f$ and $g$ are entire functions,
$d_1>0$, $d_2>0$,
$d_1/d_2\not\in \mathbb Q$,
and $|f|=|g|$ on $\cup_{k=1}^3 \ell_k$,
then   $ f\sim g $.
\end{proposition}

And the following is another characterization
of entire functions on two intersecting
lines.

\begin{proposition}[{\cite[Theorem 3.3]{Jaming2014}}]
\label{prop:p2}
Let $\theta_1, \theta_2\in \bbR$ be such that
$\theta_1-\theta_2 \not\in \mathbb Q\pi$.
Suppose that $f,g$ are entire functions of finite order
and $z_0\in\mathbb C$.
If
\[
  |f(z_0+x e^{i\theta_l})| = |g(z_0+x e^{i\theta_l})|,\quad \forall x\in\bbR, l=1,2,
\]
then $f\sim g$.
\end{proposition}

To prove our results,
we also need the following density result
of zeros of entire functions.

\begin{proposition}[{\cite[Theorem 2.5.13]{Boas1954}}] \label{prop:pa}
Let $f$ be a nonzero  entire function of
positive  order $\varrho$ and finite type $\tau$.
Denote by $n(r)$ the number of zeros of $f$ in the area
$\{z:\, |z|\le r\}$. Then
\[
  \liminf_{r\to\infty} \frac{n(r)}{r^{\varrho}}
  \le \varrho\tau,  \qquad
  \limsup_{r\to\infty} \frac{n(r)}{r^{\varrho}}
  \le e\varrho\tau.
\]
\end{proposition}

\begin{corollary}\label{cor:zero-count}
Let $m\ge1$ be an integer and let $f$ be entire with
$|f(z)|\le Ce^{A|z|^m}$ for some $C,A>0$.
If $f\not\equiv0$, there exist constants $B,r_0>0$ such that
\begin{equation}\label{eq:e1}
 n_f(r)\le Br^m,\qquad r\ge r_0,
\end{equation}
where $n_f(r)$ counts the zeros of $f$ in $|z|\le r$ with
multiplicities.

Consequently, if $\{\lambda_n:n\ge1\}$ is a strictly
increasing sequence of positive numbers such that
\[
 \lim_{n\to\infty}\frac{\lambda_n}{n^{1/m}}=0
 \quad \mbox { and }\quad  f(\lambda_k)=0, \,\, \forall k\ge 1,
\]
then $f\equiv0$. In particular, when $m=2$, this applies to
$\lambda_n=a n^\nu$ for every $a>0$ and $0<\nu<1/2$.
\end{corollary}

\begin{proof}
Suppose that $f\not\equiv0$. Its order is at most $m$.
If its order is $m$, the growth bound gives finite type, and
Proposition~\ref{prop:pa} gives the asserted bound (\ref{eq:e1}).
If its order is smaller than $m$, the function $e^{z^m}f(z)$
has order $m$ and finite type, with the same zeros as $f$.
Applying the same proposition to this function gives the same bound.

Under the additional assumptions, the identity theorem implies
$\lambda_n\to\infty$. For all sufficiently large $n$, we therefore have
\[
 1\le\frac{n_f(\lambda_n)}n
   \le B\left(\frac{\lambda_n}{n^{1/m}}\right)^m.
\]
The right-hand side tends to zero as $n\to\infty$, a contradiction.
\end{proof}

Combining Corollary~\ref{cor:zero-count} with
Proposition~\ref{prop:levin-sector}, we obtain the following
uniqueness criterion for zeros on the positive real axis.

\begin{lemma}\label{lem:levin-ray}
Let $m\ge1$ be an integer and let $f$ be entire with
$|f(z)|\le C e^{A|z|^m}$ for some $C,A>0$.
Suppose that
$\{\lambda_n:\, n\ge 1\}$ is a strictly  increasing sequence of positive
numbers such that
\[
 \frac{\lambda_n}{n^{1/m}}\longrightarrow a\in[0,\infty)
 \quad \mbox { and }\quad  f(\lambda_k)=0, \,\, \forall k\ge 1.
\]
If $a=0$, then $f\equiv0$. If $a>0$ and
\begin{equation}\label{eq:levin-ray}
 H_f^{[m]}\left(-\frac{\pi}{2m}\right)
 +H_f^{[m]}\left(\frac{\pi}{2m}\right)
 <\frac{2\pi}{a^m},
\end{equation}
then $f\equiv0$ as well.

For $a>0$, the same conclusion holds if $f$ is holomorphic
in a domain containing the closed sector
\[
 S_m=\{re^{i\theta}:r\ge0,\ |\theta|\le\pi/(2m)\}
\]
and the stated growth bound holds on $S_m$, with
$H_f^{[m]}$ defined by the same formula.
\end{lemma}

\begin{proof}
The case $a=0$ follows from Corollary~\ref{cor:zero-count}.
Suppose now that $a>0$ and $f\not\equiv0$, and put
$\Xi=\{\lambda_n:n\ge1\}$,
$n_\Xi(r)=\#\{n:\lambda_n\le r\}$.
The assumption $\lambda_n/n^{1/m}\to a$ gives
\[
 \lim_{r\to\infty}\frac{n_\Xi(r)}{r^m}=\frac{1}{a^m}.
\]
Since $\Xi$ lies on the positive real axis, the angular-density
convention in Section~\ref{sec:preliminaries} gives
\[
 \rmd\Delta_\Xi=\frac{\delta_0}{a^m} ,
\]
here $\delta_0$ denotes the Dirac measure at $\theta=0$
in the angular variable.
We apply Proposition~\ref{prop:levin-sector} with
$\varrho(r)\equiv m$ and $\alpha=-\pi/(2m)$.
The growth assumption ensures that $f$ is at most of normal
type relative to this proximate order, and the corresponding
indicator is $H_f^{[m]}$.
The integral-order formula for $H_\Xi$ yields
\[
 H_\Xi\left(-\frac{\pi}{2m}\right)
 =H_\Xi\left(\frac{\pi}{2m}\right)
 =\frac{\pi}{a^m}.
\]
Therefore Proposition~\ref{prop:levin-sector} gives
\[
 H_f^{[m]}\left(-\frac{\pi}{2m}\right)
 +H_f^{[m]}\left(\frac{\pi}{2m}\right)
 \ge \frac{2\pi}{a^m},
\]
contrary to \eqref{eq:levin-ray}.

For $a>0$, the proof uses only holomorphy and the growth
bound on $S_m$, which also proves the last assertion.
\end{proof}

The following change of basis relates equality of moduli to a
real-value condition.

\begin{lemma}\label{Lm:L1}
Let $F,G$ be entire functions, not both identically zero, and let
$|c_1|=|c_2|=1$ with $c_1/c_2\notin\mathbb R$.
Define the entire functions $U,V$ uniquely by
\begin{equation}\label{Eq2}
 F=c_1U+c_2V,\qquad G=\bar c_1U+\bar c_2V.
\end{equation}
Then, at every $z\in\mathbb C$,
\begin{equation}\label{eq:modulus:real}
 |F(z)|=|G(z)|
 \quad\Longleftrightarrow\quad
 \operatorname{Im}\bigl(U(z)\overline{V(z)}\bigr)=0.
\end{equation}
 Moreover, $F,G$ are linearly independent
exactly when $U,V$ are.
\end{lemma}
\begin{proof}
Since $|c_2|=1$ and $c_1/c_2\notin\mathbb R$, we have
\[
 c_1\bar c_2-c_2\bar c_1
   =2i\operatorname{Im}(c_1\bar c_2)
   =2i\operatorname{Im}(c_1/c_2)\ne0.
\]
Solving \eqref{Eq2} gives
\[
 U=\frac{\bar c_2 F-c_2G}{c_1\bar c_2-c_2\bar c_1},
 \qquad
 V=\frac{c_1G-\bar c_1 F}{c_1\bar c_2-c_2\bar c_1}.
\]
Thus $U,V$ are uniquely defined entire functions.

Using $|c_1|=|c_2|=1$, direct expansion yields
\[
 |F|^2-|G|^2
 =-4\operatorname{Im}(c_1\bar c_2)
       \operatorname{Im}(U\overline V).
\]
Since $\operatorname{Im}(c_1\bar c_2)\ne0$, this proves
\eqref{eq:modulus:real}.

Finally, the change of basis in \eqref{Eq2} is invertible and
therefore preserves linear independence over $\mathbb C$.
Hence $F,G$ are linearly independent if and only if $U,V$ are.
\end{proof}

The preceding change of basis gives the following characterization of Gabor phase retrieval in terms of Fock functions.

\begin{lemma}\label{Lm:L5}
Let $\Lambda\subset\mathbb R^2$ be a countable sequence. The following are equivalent:
\begin{enumerate}
\item $\Lambda$ does not do Gabor phase retrieval on $L^2(\mathbb R)$;
\item there are $U,V\in F^2(\mathbb C)$, linearly independent over
$\mathbb C$, such that
$V(\lambda)\overline{U(\lambda)}\in\mathbb R$ for all $\lambda\in\Gamma_\Lambda^*$.
\end{enumerate}

In particular, if $p_1,p_2\in L^2(\mathbb R)$, $p_1$ is not a zero function and
$P_{p_2}=Q P_{p_1}$ for a nonconstant entire function $Q$
real on $\Gamma_\Lambda^*$, then $\Lambda$ does not do Gabor phase retrieval.

Equivalently, $\Lambda$ does Gabor phase retrieval if and only if,
for every $U,V\in F^2(\mathbb C)$ with $U\not\equiv0$,
\[
 V(\lambda)\overline{U(\lambda)}\in\mathbb R
 \quad(\lambda\in\Gamma_\Lambda^*)
 \quad\Longrightarrow\quad
 V=cU\quad\text{for some }c\in\mathbb R.
\]
\end{lemma}
\begin{proof}
By \eqref{eq:vphi:P}, $f,g\in L^2(\mathbb R)$ have equal Gabor magnitudes on $\Lambda$ if and only if $|P_f|=|P_g|$ on $\Gamma_\Lambda^*$. The map $f\mapsto P_f$ is a linear bijection from
$L^2(\mathbb R)$ onto $F^2(\mathbb C)$, and its linearity and
injectivity give $f\sim g$ if and only if $P_f\sim P_g$.
Thus (i) is equivalent to the existence of
$F,G\in F^2(\mathbb C)$ with $F\nsim G$ and $|F|=|G|$ on $\Gamma_\Lambda^*$.

Given (ii), the functions $F=U+iV$, $G=U-iV$ have equal sampled
moduli and $F\nsim G$, by Lemma~\ref{Lm:L1}.
Conversely, suppose $F\nsim G$ have equal sampled moduli.
If they are independent, take $U=(F+G)/2$, $V=(F-G)/(2i)$.
If they are dependent, write $F=aH$, $G=bH$ with $H\ne0$ and
$|a|\ne|b|$, then $H$ vanishes on $\Gamma_\Lambda^*$. Take $U=H$
and choose $V\in F^2(\mathbb C)$ independent of $H$.
In this case the sampled product vanishes, so (ii) again holds.
This proves (i)$\Longleftrightarrow$(ii), including the zero-function cases.

For the entire-multiplier assertion, $P_{p_1}$ and $P_{p_2}$ are linearly independent because $Q$ is nonconstant and $P_{p_1}\not\equiv0$. At every $\lambda\in\Gamma_\Lambda^*$,
\[
P_{p_2}(\lambda)\overline{P_{p_1}(\lambda)}=Q(\lambda)|P_{p_1}(\lambda)|^2\in \mathbb R.
\]
Hence (ii) gives the conclusion.

Finally, suppose that $\Lambda$ does Gabor phase retrieval.
Let $U,V\in F^2(\mathbb C)$ satisfy $U\not\equiv0$ and the
sampled real-value condition. By the equivalence of (i) and
(ii), they are linearly dependent, so $V=cU$ for some
$c\in\mathbb C$. There exists $\lambda_0\in\Gamma_\Lambda^*$ with
$U(\lambda_0)\ne0$, otherwise $U$ and the zero function
would have identical sampled moduli, contradicting phase
retrieval in the Fock space.  Thus
\[
 c\,|U(\lambda_0)|^2
 =V(\lambda_0)\overline{U(\lambda_0)}\in\mathbb R,
\]
which gives $c\in\mathbb R$.
Conversely, if every such pair satisfies $V=cU$ with
$c\in\mathbb R$, then (ii) cannot hold, and hence
$\Lambda$ does Gabor phase retrieval.
\end{proof}

The same proof applies to evaluation phase retrieval on a countable
set in any complex vector space of entire functions of dimension
at least two. Failure is equivalent to an independent pair with
real sampled product.

To apply Lemma~\ref{Lm:L5} in constructing counterexamples,
we first establish the following multiplier result.

\begin{lemma}\label{Lm10}
Let $ G(z)$ be a non-constant entire function
which meets one of the following conditions,
\begin{enumerate}

\item  the order $\varrho$ of $G$ is less than  $2$;

\item the order $ \varrho$ is equal to $2$ and the type $\tau$ is less than $\pi/2 $;

\item $G(z) = e^{c \eta   (z+z_0)^2}$, where  $0<c<\pi$, $z_0,\eta \in\mathbb C$
    and $|\eta |=1$.
\end{enumerate}
Then there exist two functions $ p_1$ and $p_2\in L^2(\mathbb{R}) $ which are linearly independent
such that
\begin{equation}\label{eq:G hp}
P_{p_2}(z)=G(z) P_{p_1}(z).
\end{equation}
\end{lemma}

\begin{proof}
Since $f\mapsto P_f$ is unitary, it suffices to construct linearly independent
$F_1,F_2\in F^2(\mathbb C)$ with $F_2=GF_1$.

Under (i) or (ii), there are $C>0$ and $T<\pi/2$ such that
$|G(z)|\le Ce^{T|z|^2}$. Hence $G\in F^2(\mathbb C)$, and we
may take $F_1=1$, $F_2=G$.
Under (iii), take
\[
 F_1(z)=e^{-c\eta(z+z_0)^2/2},\qquad
 F_2(z)=e^{c\eta(z+z_0)^2/2}.
\]
Then $F_2=GF_1$, and, for $j=1,2$,
\[
 |F_j(z)|^2e^{-\pi|z|^2}
 \le C e^{-(\pi-c)|z|^2+2c|z_0||z|},
\]
which is integrable because $c<\pi$.
In either case $F_1\not\equiv0$ and nonconstancy of $G$ gives
linear independence. Choose $p_j\in L^2(\mathbb R)$ with
$P_{p_j}=F_j$. Then \eqref{eq:G hp} holds, and $p_1, p_2$ are linearly independent by unitarity.
\end{proof}

The following necessary condition will be used in the
nonuniqueness assertions.

\begin{lemma}\label{Th4}
Let $\Lambda\subset\mathbb R^2$ have no accumulation points.
If $\Lambda$ does Gabor phase retrieval, then
\[
  L_\Lambda\ge 2.
\]
\end{lemma}

\begin{proof}
We prove the contrapositive. First suppose that $\Lambda$ is finite,
including the possibility that it is empty. The nonconstant polynomial
\[
  P(z)=z\prod_{\zeta\in\Gamma^*_\Lambda}(z-\zeta)
\]
vanishes on $\Gamma^*_\Lambda$ and has order less than $2$.

Now suppose that $\Lambda$ is infinite and $L_\Lambda<2$.
Enumerate the nonzero points of $\Gamma^*_\Lambda$ as $\{z_n:n\ge1\}$,
let $p$ be their genus, and set
\[
  P(z)=z^\delta\prod_{n=1}^\infty E_p(z/z_n),
  \qquad
  \delta=\begin{cases}1,&0\in\Gamma^*_\Lambda,\\
                       0,&0\notin\Gamma^*_\Lambda.
          \end{cases}
\]
By Proposition~\ref{Prop5} and~\eqref{eq:rho}, this is a nonconstant
entire function of order $L_\Lambda<2$ that vanishes on
$\Gamma^*_\Lambda$.

Since $P$ has order less than $2$, it belongs to $F^2(\mathbb C)$.
Choose $0\ne f\in L^2(\mathbb R)$ with $P_f=P$.
By \eqref{eq:vphi:P}, $V_\varphi f$ vanishes on $\Lambda$.
Thus $f$ and the zero function have identical sampled magnitudes,
so $\Lambda$ does not do Gabor phase retrieval.
\end{proof}

\subsection{Proof for two intersecting lines}

Note that
Jaming's
complete-line theorem also allows suitable uniqueness sets on
the lines; see \cite[Remark~3.4]{Jaming2014}.

To prove the negative part of Theorem~\ref{thm:t2a},
we need the following lemma.

\begin{lemma}\label{Lm:L3}
Let $\Lambda = \cup_{l=1}^2 \big((t_0,\omega_0)+\bbR(\sin\theta_l,\cos\theta_l)\big)$, where $(t_0, \omega_0) \in \bbR^2$ and $\theta_1 -\theta_2\in \mathbb Q\pi$. Then there is an entire function $ Q $ with order less than $2 $ such that
$Q$ is not a constant and takes real values on $\Gamma^*_\Lambda$.
\end{lemma}

\begin{proof}
Put $z_0=\omega_0+it_0$ and write
$\theta_1-\theta_2=p\pi/q$, where $p$ is an integer and $q$ is a
positive integer.
The nonconstant polynomial
\begin{equation*}
 Q(z)=\bigl((z-z_0)e^{-i\theta_2}\bigr)^q
\end{equation*}
has order zero. For every $r\in\mathbb R$,
\[
 Q(z_0+re^{i\theta_2})=r^q,\qquad
 Q(z_0+re^{i\theta_1})=(-1)^p r^q.
\]
Thus $Q$ is real on $\Gamma^*_\Lambda$.
\end{proof}

We are now ready to give a proof of Theorem~\ref{thm:t2a}.

\begin{proof}[Proof of Theorem~\ref{thm:t2a}]
Fix $f\in L^2(\mathbb R)$ and $\theta\in\mathbb R$.
Set $z_0=\omega_0+it_0$ and define
\[
F_{f,\theta}(z):= P_f(z_0+ze^{i\theta})\overline{P_f(z_0+\overline{z}e^{i\theta})}.
\]
This is entire, and $F_{f,\theta}(x)=|P_f(z_0+xe^{i\theta})|^2$ for real $x$.
By \eqref{eq:Pf:growth},
\begin{align}
 |F_{f,\theta}(z)|
 &\le\|f\|_2^2\exp\left\{
       \frac\pi2\bigl(|z_0+ze^{i\theta}|^2
                    +|z_0+\bar z e^{i\theta}|^2\bigr)\right\}
       \nonumber\\
 &=\|f\|_2^2\exp\left\{
       \pi|z|^2+\pi|z_0|^2
       +2\pi\operatorname{Re}(z_0e^{-i\theta})\operatorname{Re}z
       \right\}.
       \label{eq:ea3}
\end{align}
Thus $F_{f,\theta}$ has finite quadratic type, and
\[
 H_{F_{f,\theta}}^{[2]}(\alpha)\le\pi,
 \qquad \alpha\in\mathbb R.
\]

Suppose that $f,g\in L^2(\mathbb R)$ have identical sampled
Gabor magnitudes on $\Lambda$, and put
$D_l=F_{f,\theta_l}-F_{g,\theta_l}$, $l=1,2$.
Equation~\eqref{eq:vphi:P} implies that $D_l$ vanishes on
$a\mathbb Z^\nu$. The same bound \eqref{eq:ea3} holds for $D_l$
with $\|f\|_2^2$ replaced by $\|f\|_2^2+\|g\|_2^2$.
If $\nu<1/2$, Corollary~\ref{cor:zero-count} with $m=2$ and
$\lambda_n=a n^\nu$ gives $D_l\equiv0$.
If $\nu=1/2$ and $a<1$, then $D_l(a\sqrt n)=0$ and
\[
 H_{D_l}^{[2]}(-\pi/4)+H_{D_l}^{[2]}(\pi/4)
 \le2\pi<\frac{2\pi}{a^2}.
\]
Lemma~\ref{lem:levin-ray} with $m=2$ gives $D_l\equiv0$.
In either case $D_1=D_2=0$, and hence
\[
 |P_f(z_0+xe^{i\theta_l})|
 =|P_g(z_0+xe^{i\theta_l})|,
 \qquad x\in\mathbb R,\quad l=1,2.
\]
If either entire function is zero, both are zero by the identity
theorem. Otherwise they have finite order, and the irrational-angle
assumption and Proposition~\ref{prop:p2} give $P_f\sim P_g$.
Injectivity of $f\mapsto P_f$ then gives $f\sim g$.

Next we prove the negative part.
There are five cases.

(i)\, $(\theta_1 -\theta_2)/\pi \in \mathbb Q$.

By Lemma~\ref{Lm:L3}, there is some entire function
$Q$ with order less than $2$ which takes real values on $\Gamma^*_\Lambda$.
And by Lemma~\ref{Lm10}, there exist functions
$p_1, p_2\in L^2(\bbR)$ such that
\[
P_{p_2}(z)=Q(z)P_{p_1}(z),\quad z\in\mathbb C.
\]
We see from Lemma~\ref{Lm:L5} that
$\Lambda$ does not do Gabor phase retrieval.

(ii) \, $\nu>1/2$.

For this finite union of sampled lines, $L_\Lambda=1/\nu<2$.
Lemma~\ref{Th4} therefore gives nonuniqueness.

(iii) \, $\nu=1/2$, $a>1$ and $|\cos (\theta_1-\theta_2)|>1/a^2$.

By Lemma~\ref{Lm10}, for any $0<c<\pi$,
there exist   functions    $p_1,  p_2\in L^2(\mathbb{R}) $ such that
\[
      P_{p_2}(z) = e^{c e^{i\alpha} (z-z_0)^2} P_{p_1}(z),
      \quad \forall z\in\mathbb C.
\]
By Lemma~\ref{Lm:L5},
 it suffices to show that for some $c$,
$Q(z):=e^{c e^{i\alpha} (z-z_0)^2}$ takes real values on $\Gamma_{\Lambda}^*$.

Since $|\cos (\theta_1-\theta_2)|>1/a^2$, there is some $0<c<\pi$
such that
\[
  c a^2 |\cos (\theta_1-\theta_2)| = \pi.
\]
Set $\alpha = \pi/2-\theta_1-\theta_2$.
For $z=z_0 \pm ak^{1/2}e^{i\theta_1}$, we have
\begin{align*}
  Q(z) = \exp\{c  e^{i\alpha}  a^2k e^{2i\theta_1}\}
  &= \exp\{c a^2k (\cos(\alpha+2\theta_1)+i\sin(\alpha+2\theta_1))  \}\\
 & = \exp\{c a^2k (\cos(\alpha+2\theta_1)+i\cos(\theta_1-\theta_2))  \}.
\end{align*}
And for
$z=z_0 \pm ak^{1/2}e^{i\theta_2}$, we have
\begin{align*}
  Q(z) = \exp\{c  e^{i\alpha}  a^2k e^{2i\theta_2}\}
 & = \exp\{c a^2k (\cos(\alpha+2\theta_2)+i\cos(\theta_2-\theta_1))  \}.
\end{align*}
In both cases, we have $Q(z)\in\bbR$.

(iv) \, $\nu=1/2$, $a>1$ and $|\sin (\theta_1-\theta_2)|>1/a^2$.

Let $p_1$, $p_2$ and $Q$ be defined as in Case (iii).
Since $|\sin (\theta_1-\theta_2)|>1/a^2$, there is some $0<c<\pi$
such that
\[
  c a^2 |\sin (\theta_1-\theta_2)| = \pi.
\]
Set $\alpha = -\theta_1-\theta_2$.
For $z=z_0 \pm ak^{1/2}e^{i\theta_1}$, we have
\begin{align*}
  Q(z) = \exp\{c  e^{i\alpha}  a^2k e^{2i\theta_1}\}
  &= \exp\{c a^2k (\cos(\alpha+2\theta_1)+i\sin(\alpha+2\theta_1))  \}\\
 & = \exp\{c a^2k (\cos(\alpha+2\theta_1)+i\sin(\theta_1-\theta_2))  \}.
\end{align*}
And for
$z=z_0 \pm ak^{1/2}e^{i\theta_2}$, we have
\begin{align*}
  Q(z) = \exp\{c  e^{i\alpha}  a^2k e^{2i\theta_2}\}
 & = \exp\{c a^2k (\cos(\alpha+2\theta_2)+i\sin(\theta_2-\theta_1))  \}.
\end{align*}
In both cases, we have $Q(z)\in\bbR$.

(v) \, $\nu=1/2$, $a>1$ and  $|\sin(2\theta_1-2\theta_2)|>1/a^2$.

Let $p_1$, $p_2$ and $Q$ be defined as in Case (iii).
Since $|\sin(2\theta_1-2\theta_2)|>1/a^2$, there is some $0<c<\pi$
such that
\[
  c a^2 |\sin (2\theta_1-2\theta_2)| = \pi.
\]
Set $\alpha = -2\theta_2$.
For $z=z_0 \pm ak^{1/2}e^{i\theta_1}$, we have
\begin{align*}
  Q(z) = \exp\{c  e^{i\alpha}  a^2k e^{2i\theta_1}\}
  &= \exp\{c a^2k (\cos(\alpha+2\theta_1)+i\sin(\alpha+2\theta_1))  \}\\
 & = \exp\{c a^2k (\cos(\alpha+2\theta_1)+i\sin(2\theta_1-2\theta_2))  \}.
\end{align*}
And for
$z=z_0 \pm ak^{1/2}e^{i\theta_2}$, we have
\begin{align*}
  Q(z) = \exp\{c  e^{i\alpha}  a^2k e^{2i\theta_2}\}
 & = \exp\{c a^2k  \cos(\alpha+2\theta_2)  \}.
\end{align*}
In both cases, we have $Q(z)\in\bbR$.
This completes the proof.
\end{proof}

\subsection{Proof for three parallel lines}

Wellershoff's theorem, Proposition~\ref{Prop3},
gives uniqueness on three complete lines with irrational distance
ratio. Liehr \cite{Liehr2024} more generally characterizes
phase retrieval on parallel lines by the absence of an arithmetic
progression containing their transverse coordinates.

The following lemma is used in the proof of
Theorem~\ref{thm:t2b}.

\begin{lemma} \label{Lm:L4}
Let $\Lambda = \cup_{l=1}^3 \big((t_l,\omega_l)+\bbR(\sin\theta_0,\cos\theta_0) \big)$, where $\theta_0 \in [0,2\pi]$, $(t_l, \omega_l) \in \bbR^2$, $l=1,2,3$,  $d_l:=|(\omega_l-\omega_3)\sin\theta_0
- (t_l-t_3)\cos\theta_0|>0$, $l=1,2$
and $d_1/d_2\in  \mathbb Q$. Then there is an
entire function $ Q $ with order less than $2$
such that $Q$ is not a constant and takes real values on $\Gamma^*_\Lambda$.
\end{lemma}

\begin{proof}
Denote $ z_l=\omega_l+it_l$, $l=1,2,3$.
We have  $\Gamma^*_\Lambda=\cup_{l=1}^3 (z_l+\bbR e^{i\theta_0})$.
Suppose that
${d_1}/{d_2}=n/m $,
where $m$, $n$ are positive integers.

Let
\[
    Q(z):=\exp\Big\{\frac{n\pi (z-z_3)e^{-i\theta_0}}{d_1}\Big\}.
\]
It is easy to see that  $ Q(z) $ is an entire
function with order $1$.
Let us prove that $Q$ takes real values on $\Gamma_{\Lambda}^*$.
There are three cases.

(i) $ z= z_1+re^{i\theta_0},\ r\in\bbR $.

In this case, we have
\[
\frac{(z_1-z_3)e^{-i\theta_0}}{d_1}
=\frac{(\omega_1-\omega_3)\cos\theta_0
+(t_1-t_3)\sin\theta_0+i(t_1-t_3)\cos\theta_0
-i(\omega_1-\omega_3)\sin\theta_0}{d_1}.
\]
Note that
\[
\frac{i(t_1-t_3)\cos\theta_0-i(\omega_1-\omega_3)\sin\theta_0}{d_1}
=\pm i.
\]
We have
\[
    Q(z)= \exp\{\frac{n\pi r}{d_1}\}\exp\{\frac{n\pi(z_1-z_3)
     e^{-i\theta_0}}{d_1} \} \in \bbR,\ \forall  r\in\bbR .
 \]

(ii)\,   $ z=z_2+re^{i\theta_0}$,
$r\in\bbR $.

We have
\[
\frac{(z_2-z_3)e^{-i\theta_0}}{d_1}
=\frac{(\omega_2-\omega_3)\cos\theta_0+(t_2-t_3)
 \sin\theta_0+i(t_2-t_3)\cos\theta_0
 -i(\omega_2-\omega_3)\sin\theta_0}{d_1}.
\]
Note  that
\[
\frac{i(t_2-t_3)\cos\theta_0-i(\omega_2-\omega_3)
 \sin\theta_0}{d_1} = \pm \frac{m}{n}i .
\]
We have
\[
    Q(z)=\exp\{\frac{n\pi r}{d_1}\}\exp\{\frac{n\pi(z_2-z_3)e^{-i\theta_0}}{d_1} \}\in \bbR,\ \forall r\in\mathbb{R}.
 \]

 (iii) \, $ z=z_3+re^{i\theta_0}$, $ r\in\bbR $.

In this case,
\[
    Q(z)=\exp\{\frac{n\pi r}{d_1}\}\in\mathbb{R},\ \forall r\in\mathbb{R}.
\]
This completes the proof.
\end{proof}

\begin{proof}[Proof of Theorem~\ref{thm:t2b}]
Applying Proposition~\ref{Prop3} instead of
Proposition~\ref{prop:p2},
with similar arguments as in the proof of  Theorem~\ref{thm:t2a} we obtain that
$\Lambda$ does Gabor phase retrieval
if
$d_1/d_2\not\in  \mathbb Q$
and either $\nu<1/2$ or $\nu=1/2$ with $a<1$.

For the negative part, there are three cases.

First, we see from Lemmas \ref{Lm:L5}, \ref{Lm10}  and \ref{Lm:L4} that
$\Lambda$ does  not do Gabor phase retrieval
if  $d_1/d_2 \in  \mathbb Q$.
This proves (i).

For (ii), the finite union of sampled lines has
$L_\Lambda=1/\nu<2$, so Lemma~\ref{Th4} gives nonuniqueness.
It remains to show that
$\Lambda$ does not do Gabor phase retrieval
if it meets (iii).

Put $z_l=\omega_l+it_l$, $l=1,2,3$, and set
$Q(z)=e^{ic e^{-2i\theta_0}(z-z_3)^2}$.
By Lemma~\ref{Lm10}, for $c=\pi/a^2$,
there exist   functions    $p_1,  p_2\in L^2(\mathbb{R}) $ such that
\[
      P_{p_2}(z) = Q(z) P_{p_1}(z),
      \quad \forall z\in\mathbb C.
\]
By Lemma~\ref{Lm:L5},  it suffices to show that
$Q(z) $ takes real values on $\Gamma_{\Lambda}^*$.

For $z=z_1 \pm ak^{1/2} e^{i\theta_0}$,
we have $z-z_3 = d_1 e^{i(\theta_0\pm \pi/2)}\pm ak^{1/2} e^{i\theta_0}$.
Hence
\begin{align*}
Q(z) &= \exp\{ic e^{-2i\theta_0}(d_1 e^{i(\theta_0\pm \pi/2)}\pm ak^{1/2} e^{i\theta_0})^2 \} \\
&= \exp\{ic(a^2k-d_1^2)\pm 2acd_1 k^{1/2}\}\\
&\in \bbR.
\end{align*}
Similarly we obtain that
$Q(z)\in\bbR$ for $z=z_2\pm ak^{1/2} e^{i\theta_0}$
or $z=z_3\pm ak^{1/2} e^{i\theta_0}$.
This completes the proof.
\end{proof}

\section{square-root lattices}\label{sec:square-root}

In this section, we prove Theorem~\ref{thm:t1} and Corollary~\ref{coro1.2}.
The following lemma gives its nonuniqueness assertion.

\begin{lemma} \label{Lm:L6}
Let $ \Lambda=a\mathbb{Z}^{1/2}\times b\mathbb{Z}^{1/2} $,
where
 $a$, $b$ and $c $  are positive numbers
satisfying   $  a^2c\in\pi\mathbb{Z}$, $b^2c\in\pi\mathbb{Z}$,
and $0<c< \pi  $.
Then $\Lambda$ does not do Gabor phase retrieval.
\end{lemma}

\begin{proof}
Put $Q(z)=e^{-icz^2}$. Lemma~\ref{Lm10}(iii) gives
$p_1,p_2\in L^2(\mathbb R)$, with $p_1\ne0$, such that
\begin{equation*}
 P_{p_2}(z)=Q(z)P_{p_1}(z),\qquad z\in\mathbb C.
\end{equation*}
For $z=x+iy=\pm b\sqrt m\pm ia\sqrt n$,
\[
 Q(z)=e^{-ic(b^2m-a^2n)}e^{2cxy}\in\mathbb R,
\]
thanks to $a^2c,b^2c\in\pi\mathbb Z$.
The conclusion follows from Lemma~\ref{Lm:L5}.
\end{proof}

For the uniqueness assertion, we first derive the arithmetic
restriction from first-quadrant samples. We then use the symmetric lattice to obtain the sharp
growth bound and classify its equality case.

\begin{lemma}\label{lem:arithmetic:disk}
Let $f_1,\ldots,f_J$ be nonzero entire functions satisfying
$|f_j(z)|\le C_j e^{T_j|z|^2}$, with $C_j>0$ and $T_j\ge0$.
Fix $\varepsilon>0$.
For every sufficiently large $R$, the square
\[
 \{z:R<\operatorname{Re}z<2R,\ R<\operatorname{Im}z<2R\}
\]
contains a closed disk of radius $r=R^{-\varepsilon}$ on which all
$f_j$ are nonvanishing and
\[
 \sum_{j=1}^J\left|\frac{f_j'}{f_j}\right|\le CR,
 \qquad
 \sum_{j=1}^J\left|\left(\frac{f_j'}{f_j}\right)'\right|
 \le C\log R.
\]
The constant $C$ is independent of $R$.
\end{lemma}
\begin{proof}
Write $f_j(z)=z^{m_j}g_j(z)$, where $g_j(0)\ne0$.
Corollary~\ref{cor:zero-count} with $m=2$ gives
$n_{f_j}(t)\lesssim t^2$, where zeros are counted with multiplicities.
List the zeros of $f_j$ in $|z|\le8R$ as
$\alpha_{j,1},\ldots,\alpha_{j,N_j}$, and put
\[
 P_{j,R}(z)=\prod_{k=1}^{N_j}\frac{z-\alpha_{j,k}}{8R},
 \qquad G_{j,R}=\frac{f_j}{P_{j,R}}.
\]
An empty product is $1$.
Then $G_{j,R}$ is entire and has no zero in $|z|\le8R$.
On $|z|=16R$, every factor of $P_{j,R}$ has modulus at least $1$.
The maximum principle therefore gives
\[
 |G_{j,R}(z)|\le C_j e^{256T_jR^2},\qquad |z|\le16R.
\]
Moreover, for $R\ge1$,
\[
 |G_{j,R}(0)|
 =(8R)^{m_j}|g_j(0)|
       \prod_{\alpha_{j,k}\ne0}\frac{8R}{|\alpha_{j,k}|}
 \ge |g_j(0)|.
\]
Choose a holomorphic logarithm $H_{j,R}$ of $G_{j,R}$ on a
neighborhood of the closed disk $|z|\le8R$, that is $G_{j,R}(z)=e^{H_{j,R}(z)}$.
Let
\[
 a_j=256T_j+|\log C_j|+1,\qquad b_j=\log|g_j(0)|,
\]
we have $\operatorname{Re}H_{j,R}\le a_jR^2$ on that disk and
$\operatorname{Re}H_{j,R}(0)\ge b_j$. The
Borel--Carath\'eodory inequality
\cite[Theorem 1.6.1]{Ramachandra1995} gives
\[
 \max_{|z|\le6R}|H_{j,R}(z)-H_{j,R}(0)|
 \le6(a_jR^2-b_j)\le B_jR^2,
 \qquad B_j=6(a_j+|b_j|).
\]
For $|z|\le4R$, apply Cauchy's estimates to
$H_{j,R}-H_{j,R}(0)$ on the circle of radius $2R$ about $z$.
This yields
\begin{equation}\label{eq:arithmetic:harmonicderivatives}
 |H_{j,R}'(z)|\le\tfrac12B_jR,
 \qquad |H_{j,R}''(z)|\le\tfrac12B_j.
\end{equation}

Collect the zero lists into $\beta_1,\ldots,\beta_N$, retaining
multiplicities.
Since $n_{f_j}(t)\lesssim t^2$, we have $N\le C_0R^2$, for a constant $C_0$.
Put
\[
 S_\ell(z)=\sum_{k=1}^N|z-\beta_k|^{-\ell},\qquad \ell=1,2;
\]
both sums are zero when $N=0$.
Let
\[
 K_R=\{x+iy:\frac{5R}{4}<x,y<\frac{7R}{4}\},\qquad
 E_R=K_R\setminus\bigcup_{k=1}^ND(\beta_k,2r).
\]
We use $|E|$ for planar Lebesgue measure and $\rmd A$ for its
integration element. The removed area is at most
$4\pi Nr^2\le4\pi C_0R^{2-2\varepsilon}$; hence
$|E_R|\ge R^2/8$ for all sufficiently large $R$.
For $z\in K_R$ and every listed zero, $|z-\beta_k|<12R$.
For $R\ge2$, Polar integration and summation give
\begin{align*}
 \int_{K_R}S_1(z)\,\rmd A(z)
 &\le N\int_{|w|\le 12R}\frac{\rmd A(w)}{|w|} \le24\pi C_0R^3,\\
 \int_{E_R}S_2(z)\,\rmd A(z)
 &\le N\int_{2r\le|w|<12R} \frac{\rmd A(w)}{|w|^2} \le2\pi C_0 \left(1+\varepsilon+\frac{\log6}{\log2}\right)R^2\log R.
\end{align*}
In particular, denote $\tilde{C}=192\pi C_0+16\pi C_0\left(1+\varepsilon+\log6/\log2\right)$, then we have
\[
 \frac1{|E_R|}\int_{E_R}
       \left(\frac{S_1}{R}+\frac{S_2}{\log R}\right)\rmd A
       \le \tilde{C}
\]
 Choose $z_R\in E_R$ such that
$S_1(z_R)\le 2\tilde{C}R$ and $S_2(z_R)\le 2\tilde{C}\log R$.
For every $z\in \overline{D(z_R,r)}$, the inequality $|z_R-\beta_k|\ge2r$ gives
$|z-\beta_k|\ge |z_R-\beta_k|-r\ge |z_R-\beta_k|/2\ge r$.
Since $r<R/4$ for large
$R$, $\overline{D(z_R,r)}$ lies in the stated square and in $|z|<4R$.
Zeros outside $|z|\le8R$ cannot meet it. Hence all $f_j$ are nonvanishing on $\overline{D(z_R,r)}$.

Notice that for every $z\in \overline{D(z_R,r)}$,
\begin{equation}\label{eq:S1}
    S_1(z)\le 2\sum_{k=1}^N |z_R-\beta_k|^{-1} =2S_1(z_R)\le 4\tilde{C}R,
\end{equation}
\begin{equation}\label{eq:S2}
    S_2(z)\le 4\sum_{k=1}^N |z_R-\beta_k|^{-2}=4S_2(z_R)\le 8\tilde{C}\log R.
\end{equation}
Since
\[
 \frac{f_j'}{f_j}
 =H_{j,R}'+\sum_{k=1}^{N_j}\frac1{z-\alpha_{j,k}},
 \qquad
 \left(\frac{f_j'}{f_j}\right)'
 =H_{j,R}''-\sum_{k=1}^{N_j}\frac1{(z-\alpha_{j,k})^2},
\]
combining the above equations with \eqref{eq:arithmetic:harmonicderivatives}, \eqref{eq:S1} and \eqref{eq:S2},
we obtain that
\[
 \sum_{j=1}^J \left|\frac{f_j'}{f_j}(z)\right|\le \frac{1}{2}\sum_j B_jR+4\tilde{C}R,
\]
\[
 \sum_{j=1}^J \left|\left(\frac{f_j'}{f_j}\right)'(z)\right|
 \le \frac{1}{2}\sum_j B_j+8\tilde{C}\log R,
\]
for $\forall z\in \overline{D(z_R,r)}$,
which implies the claimed
estimates throughout the closed disk.
\end{proof}

\begin{lemma}\label{lem:arithmetic:lattice}
Let $0<h_-<h_+<\infty$ and $M>0$ be constants. Set $T=3h_++3$.
Then there are constants $c,C>0$, $0<\delta_0\le1$, and an integer
$N\ge1$, depending only on $h_-,h_+,M$,
such that for any
$U,V$ which satisfy
\begin{enumerate}
\item
$U,V$
are  holomorphic on $D(0,T)$, with both $U$ and
$W:=UV'-U'V$ nonvanishing there;
\item for some  $0\le\delta\le\delta_0$
and complex number $\lambda$ with $ |\lambda|\le M$,
\[
 \sup_{D(0,T)}|2A-\lambda|\le\delta,
\]
where
\[
 A=\frac{U'}U-\frac{W'}{2W};
\]
\item
there exist
 five points $0,w_{10},w_{20},w_{01},w_{02}$ satisfying
\[
 |w_{k0}-kh|\le\delta,\qquad |w_{0k}-ikv|\le\delta,
 \quad k=1,2,
\]
for some
$h,v\in[h_-,h_+]$
with $\operatorname{Im}(V(w)\overline{U(w)})=0$ at all five points,
\end{enumerate}
we have
  $|\lambda|\ge c$, and there are integers $(s,t)\ne(0,0)$,
$|s|,|t|\le N$, such that
\begin{equation}\label{eq:arithmetic:quantization}
 \left|\operatorname{Im}\lambda-\frac{\pi s}{h}\right|
 +\left|\operatorname{Re}\lambda-\frac{\pi t}{v}\right|
 \le C\delta.
\end{equation}
\end{lemma}

\begin{proof}
The reality condition at zero gives $q_0=V(0)/U(0)\in\mathbb R$.
Set
\[
 d=\frac{W(0)}{U(0)^2}\ne0,\qquad \beta=\frac d{|d|},
 \qquad
 F(w)=\int_0^w\exp\!\left(-2\int_0^z A(\zeta)\,\rmd\zeta\right)\rmd z.
\]
The disk is simply connected, so these integrals are path independent,
and $F$ is holomorphic on $D(0,T)$.
Since
\[
   \left( \frac{W}{U^2} \right)'=-2A\frac{W}{U^2},
\]
differentiation shows that
$(W/U^2)\exp(2\int_0^wA)$ is constant, with value $d$.
Consequently $(V/U)'=W/U^2=dF'$, and evaluation at zero gives
\[
 V=(q_0+dF)U,\qquad
 \operatorname{Im}(V\overline U)
 =|U|^2|d|\operatorname{Im}(\beta F).
\]
Thus $|\beta|=1$ and $\operatorname{Im}(\beta F)$ vanishes at
the five sampling points.

Set
\[
 L=T-1,\qquad
 \phi_\lambda(w)=\int_0^w e^{-\lambda z}\,\rmd z,
 \qquad K=(L^2+1)e^{(M+1)L}.
\]
In particular $\phi_0(w)=w$.
The actual nodes $0,w_{10},w_{20},w_{01},w_{02}$ and the model nodes
$0,h,2h,iv,2iv$ all lie in $|w|\le L$.
Since $F'(z)=\exp\!\left(-\lambda z-\int_0^z (2A(\zeta)-\lambda) \,\rmd \zeta\right)$
and
$\left|-\int_0^z (2A(\zeta)-\lambda) \,\rmd \zeta\right|\le \delta |z|$,
using $|e^u-1|\le |u|e^{|u|}$ and $\delta\le1$ gives
 $|F'(z)-e^{-\lambda z}| \le L e^{(M+1)L}\delta$, $|z|\le L$.
 Integrating along the line segment, we obtain
 \[|F(w)-\phi_\lambda(w)|\le L^2e^{(M+1)L}\delta, \qquad
 |w|\le L.\]
Also notice that $|\phi_\lambda(w_1)-\phi_\lambda(w_2)|\le e^{ML}|w_1-w_2|$,  for $\forall |w_1|,|w_2|\le L$.
Comparing each node to its ideal position yields
\begin{equation}\label{eq:arithmetic:modelerror}
 \begin{aligned}
 |F(w_{k0})-\phi_\lambda(kh)|&\le K\delta,\\
 |F(w_{0k})-\phi_\lambda(ikv)|&\le K\delta,
 \qquad k=1,2.
 \end{aligned}
\end{equation}
Furthermore,
\[
 |\phi_\lambda(w)-w|
 =\left|w\int_0^1(e^{-\lambda sw}-1)\,\rmd s\right|
 \le L^2e^{ML}|\lambda|\le K|\lambda|, \quad |w|\le L.
\]
Hence
\[
|F(w_{k0})-kh|\le K(\delta+|\lambda|),\quad |F(w_{0k})-ikv|\le K(\delta+|\lambda|),\quad k=1,2.
\]
The reality conditions at $w_{10}$ and $w_{01}$ imply
\[
 h|\operatorname{Im}\beta|=|\operatorname{Im}\big(\beta h-\beta F(w_{10})\big)|  \le K(\delta+|\lambda|),
\]
\[
 v|\operatorname{Re}\beta|=|\operatorname{Re}\big(\beta v+i\beta F(w_{01})\big)| \le K(\delta+|\lambda|).
\]
Since $h,v\ge h_-$, we have $1\le |\operatorname{Re}\beta|+|\operatorname{Im}\beta|\le(2K/h_-)(\delta+|\lambda|)$.
Define $c=h_-/(4K)$ and require $\delta_0\le c$.
Then $|\lambda|\ge c$.

We may therefore define
\[
 \gamma=\frac\beta\lambda,\qquad X=e^{-\lambda h},
 \qquad Y=e^{-i\lambda v},
\]
\[
 B=e^{Mh_+},\qquad \kappa=\frac{h_-c}{\sqrt2},
 \qquad \eta=1-e^{-\kappa}>0.
\]
Here $1/M\le|\gamma|\le1/c$ and $1/B\le|X|,|Y|\le B$.
Since $\lambda\neq0$ and $\phi_\lambda(w)=(1-e^{-\lambda w})/\lambda$,
\eqref{eq:arithmetic:modelerror} and sampled reality give
\begin{equation}\label{eq:arithmetic:fivepoint}
 \begin{aligned}
 |\operatorname{Im}\big(\gamma(1-X)\big)|&\le K\delta,&
 |\operatorname{Im}\big(\gamma(1-X^2)\big)|&\le K\delta,\\
 |\operatorname{Im}\big(\gamma(1-Y)\big)|&\le K\delta,&
 |\operatorname{Im}\big(\gamma(1-Y^2)\big)|&\le K\delta.
 \end{aligned}
\end{equation}
At least one of $|\operatorname{Re}\lambda|$ and
$|\operatorname{Im}\lambda|$ is at least $c/\sqrt2$. Since
$\log|X|=-h\operatorname{Re}\lambda$ and
$\log|Y|=v\operatorname{Im}\lambda$, at least one of
$|\log|X||,|\log|Y||$ is at least $\kappa$.

First suppose $|\log|X||\ge\kappa$. Then $|1-X|\ge\eta$,
we have $|\gamma(1-X)|\ge\eta/M$. The restriction
$\delta_0\le\eta/(2MK)$ implies
$|\operatorname{Re}(\gamma(1-X))|\ge\eta/(2M)$.
Using $1-X^2=(1-X)(1+X)$ in the second horizontal condition in \eqref{eq:arithmetic:fivepoint} gives
\[
 |\operatorname{Re}\big(\gamma(1-X)\big)|\,|\operatorname{Im}X|
 \le K\delta+|\operatorname{Im}\big(\gamma(1-X)\big)|(1+|\operatorname{Re}X|)
 \le K(B+2)\delta.
\]
Thus $|\operatorname{Im}X|\le C_X\delta$, where
\[
 C_X=\frac{2MK(B+2)}\eta.
\]
If $\delta_0\le\eta/(2C_X)$, then
$|1-\operatorname{Re}X|\ge|1-X|-|\operatorname{Im}X|\ge\eta/2$.
Expanding the first horizontal condition in \eqref{eq:arithmetic:fivepoint} now yields
\[
 |\operatorname{Im}\gamma|
 \le\frac2\eta\left(K+\frac{C_X}{c}\right)\delta
 =:C_\gamma\delta.
\]
Require $\delta_0\le1/(2MC_\gamma)$. Since $|\gamma|\ge1/M$,
we obtain $|\operatorname{Re}\gamma|\ge1/(2M)$.
The first vertical condition in \eqref{eq:arithmetic:fivepoint} consequently gives
\[
 |\operatorname{Im}Y|
 \le2M\bigl(K+(1+B)C_\gamma\bigr)\delta=:C_Y\delta.
\]
If $|\log|Y||\ge\kappa$, interchange $X$ and $Y$ throughout
this argument, using the two vertical conditions first and then
the first horizontal condition. The same constants apply.
Therefore
\[
 |\operatorname{Im}X|+|\operatorname{Im}Y|
 \le(C_X+C_Y)\delta.
\]
Since
\[
\operatorname{Im}X=-|X|\sin{(h\operatorname{Im}\lambda)},\qquad
\operatorname{Im}Y=-|Y|\sin{(v\operatorname{Re}\lambda)},
\]
the lower bounds $|X|,|Y|\ge1/B$ give
\[
 |\sin(h\operatorname{Im}\lambda)|
 +|\sin(v\operatorname{Re}\lambda)|
 \le B(C_X+C_Y)\delta.
\]
Choose $\pi s$ and $\pi t$ to be nearest integer multiples of
$\pi$ to $h\operatorname{Im}\lambda$ and
$v\operatorname{Re}\lambda$, respectively.
Their distances are at most $\pi/2$. Using
$\sin u\ge2u/\pi$ for $0\le u\le\pi/2$ proves
\eqref{eq:arithmetic:quantization} with
\[
 C=\frac{\pi B(C_X+C_Y)}{2h_-}.
\]
The nearest-integer choice also gives
$|s|,|t|\le Mh_+/\pi+1/2$, so
$N=\lceil Mh_+/\pi\rceil+1$ suffices.
Choose $\delta_0\le c/(2C)$.  If $s=t=0$, \eqref{eq:arithmetic:quantization} would imply
$|\lambda|\le C\delta\le c/2$, contradicting $|\lambda|\ge c$.
All restrictions are met by the single positive choice
\[
 \delta_0=\min\left\{1,c,\frac\eta{2MK},\frac\eta{2C_X},
             \frac1{2MC_\gamma},\frac c{2C}\right\}.
\]
Every constant depends only on $h_-,h_+,M$, as claimed.
\end{proof}

\begin{lemma}\label{lem:arithmetic:rigidity}
Let $a,b>0$. Suppose linearly independent entire functions $U,V$
satisfy $|U(z)|+|V(z)|\le C_0e^{T_0|z|^2}$ and
\[
 \operatorname{Im}\bigl(V(b\sqrt m+ia\sqrt n)
            \overline{U(b\sqrt m+ia\sqrt n)}\bigr)=0,
 \qquad m,n\ge1.
\]
Then $a^2/b^2\in\mathbb Q$.
\end{lemma}
\begin{proof}
Set $W=UV'-U'V$. If $W$ vanished identically, then $(V/U)'=0$
on any disk avoiding the zeros of $U$, and the identity theorem
would make $V$ a constant multiple of $U$. Thus $W\not\equiv0$.
Cauchy's estimate on unit circles shows that $U',V'$, and hence
$W$, have finite quadratic type. Apply
Lemma~\ref{lem:arithmetic:disk} to $U,W$ with $\varepsilon=1/4$.
For every sufficiently large $R$ it supplies
\[
 D_R=D(z_R,r_R),\qquad
 z_R=x_R+iy_R,\qquad r_R=R^{-1/4},
\]
whose closure is contained in
$\{z:R<\operatorname{Re}z<2R,\ R<\operatorname{Im}z<2R\}$
and on which $U,W$ are nonvanishing. The holomorphic function
\[
 A=\frac{U'}U-\frac{W'}{2W}
\]
satisfies, with constants independent of $R$,
\begin{equation}\label{eq:A:estimate}
 |A|\le K_0R,\qquad |A'|\le B_R:=K_1\log R\quad\hbox{on }D_R.
\end{equation}
Introduce the concentric squares
\[
 Q_R=\{x+iy:|x-x_R|,|y-y_R|\le r_R/4\},\qquad
 Q'_R=\{x+iy:|x-x_R|,|y-y_R|\le r_R/8\}.
\]
The horizontal and vertical lattice spacings in these squares are uniformly $O(R^{-1})$.
Since the side length of $Q_R$ is comparable to $r_R$ and $R^{-1}=o(r_R)$, $Q_R$ contains lattice nodes for all sufficiently large $R$.
As $Q_R$ is axis-parallel, there are positive integers $m_0\le m_1$ and $n_0\le n_1$ such that its lattice nodes are exactly
\[
\left\{b\sqrt m+ia\sqrt n:
m=m_0,\ldots,m_1,\quad n=n_0,\ldots,n_1\right\}.
\]

At any such node $z_*=x_*+iy_*=b\sqrt {m_*}+ia\sqrt{n_*}$, set
\[
 w=R(z-z_*),\qquad
 h=\frac{b^2R}{2x_*},\qquad v=\frac{a^2R}{2y_*},
 \qquad\lambda=\frac{2A(z_*)}{R}.
\]
Since $R<x_*,y_*<2R$, the fixed choices
\[
 h_- =\frac{\min\{a^2,b^2\}}8,\qquad
 h_+ =1+\frac{\max\{a^2,b^2\}}2,\qquad
 M=2K_0+1,\qquad T=3h_++3
\]
bound $h,v$ and $|\lambda|$ as required by
Lemma~\ref{lem:arithmetic:lattice}. For large $R$,
\[
 |z_*+w/R-z_R|\le\frac{\sqrt2}{4}r_R+\frac TR<r_R,
 \qquad |w|<T.
\]
Thus $D(z_*,T/R)\subset D_R$. The change of variables $w=R(z-z_*)$ maps $D(z_*, T/R)$ onto $D(0,T)$, sending the sampling point $z_*$ to origin. On $D(0,T)$ define $U_*(w):=U(z_*+w/R)$ and $V_*(w):=V(z_*+w/R)$.
With derivatives taken with respect to $w$, we have
\[
 W_*(w):=U_*(w)V_*'(w)-U_*'(w)V_*(w)=\frac1R W(z_*+w/R)
 \]
 and
 \[
 A_*(w):=\frac{U_*'}{U_*}-\frac{W_*'}{2W_*}
       =\frac1R A(z_*+w/R).
\]
In particular, $U_*,W_*$ are nonvanishing and applying \eqref{eq:A:estimate}, we obtain
\begin{equation}\label{eq:er1}
 \sup_{|w|<T}|2A_*(w)-\lambda|=\frac{2}{R}\sup_{|w|<T}|A(z_*+w/R)-A(z_*)|   \le\frac{2TB_R}{R^2}.
\end{equation}
Take the first two lattice nodes to the right of $z_*$ in the same row and the first two above $z_*$ in the same column. Under the change of variables $w=R(z-z_*)$, their coordinates are
\[
 w_{k0}=R(\sqrt{x_*^2+kb^2}-x_*),\qquad
 w_{0k}=iR(\sqrt{y_*^2+ka^2}-y_*),\qquad k=1,2.
\]
For $x,c>0$, the exact identity
\[
 \sqrt{x^2+c}-x-\frac c{2x}
 =-\frac{c^2}{2x(\sqrt{x^2+c}+x)^2}
\]
has absolute value at most $c^2/(8x^3)$. Consequently
\begin{equation}\label{eq:er2}
 |w_{k0}-kh|\le R\frac{k^2b^4}{8x_*^3}\le\frac{b^4}{2R^2},\quad
 |w_{0k}-ikv|\le R\frac{k^2a^4}{8y_*^3}\le\frac{a^4}{2R^2},\quad k=1,2.
\end{equation}
Moreover, since $|w_{k0}|\le Rkb^2/(2x_*)\le 2h_+$ and $|w_{0k}|\le 2h_+$,
these points lie in $D(0,T)$  and satisfy sampled reality.
Define
\[
 C_\delta=\max\{2T,\tfrac12\max(a^4,b^4)\},\qquad
 \delta_R=C_\delta\frac{1+B_R}{R^2}.
\]
All three approximation errors in \eqref{eq:er1} and \eqref{eq:er2} are bounded by $\delta_R\to0$.
Lemma~\ref{lem:arithmetic:lattice} therefore provides constants
$C_*,N$, independent of $R$ and of the selected node,
and a nonzero integer pair $(s,t)$ with $|s|,|t|\le N$, such that
\[
|\operatorname{Im}\lambda-\frac{\pi s}{h}|+|\operatorname{Re}\lambda-\frac{\pi t}{v}|\le C_* \delta_R.
\]
Define $P_{s,t}(x+iy)=\pi\left(\frac{ty}{a^2}+i\frac{sx}{b^2}\right)$.
Then
\begin{equation}\label{eq:arithmetic:local}
 \left|A(z_*)-P_{s,t}(z_*)\right|\le e_R,\qquad
 e_R=\frac{C_*C_\delta(1+B_R)}{2R}.
\end{equation}
The pair may initially depend on the node. We next show that
neighboring nodes must receive the same pair.

Let $z,z'$ be any horizontal or vertical neighbors in $Q_R$.
Then $|z-z'|\le\max\{a^2,b^2\}/(2R)$. Indeed,
\[
 b(\sqrt{m+1}-\sqrt m)
 =\frac{b^2}{b\sqrt{m+1}+b\sqrt m}\le\frac{b^2}{2R}
\]
and the vertical estimate is identical. The segment between them
lies in $D_R$. So
\[
 |A(z')-A(z)|\le\frac{B_R\max\{a^2,b^2\}}{2R}.
\]
For any fixed integer pair with $|s|,|t|\le N$, the corresponding
model satisfies $|P_{s,t}(z')-P_{s,t}(z)|\le\pi N/(2R)$. If the pairs at
$z,z'$ are $(s,t)$ and $(s',t')$ respectively, then
\begin{align*}
 |P_{s,t}(z)-P_{s',t'}(z)|
 &\le |P_{s,t}(z)-A(z)|+|A(z)-A(z')|\\
 &\quad+|A(z')-P_{s',t'}(z')|+|P_{s',t'}(z')-P_{s',t'}(z)|\\
 &\le2e_R+\frac{B_R\max\{a^2,b^2\}}{2R}+\frac{\pi N}{2R}.
\end{align*}
For distinct integer pairs and a point $z=x+iy$ in $Q_R$, we have
\[
 |P_{s,t}(z)-P_{s',t'}(z)|
 =\pi\sqrt{\frac{(t-t')^2y^2}{a^4}
                +\frac{(s-s')^2x^2}{b^4}}
 \ge\frac{\pi R}{\max\{a^2,b^2\}}.
\]
If the integer pairs at these neighboring nodes were distinct, the two bounds would give
\[
\frac{\pi R}{\max\{a^2,b^2\}}\le 2e_R+\frac{B_R\max\{a^2,b^2\}}{2R}+\frac{\pi N}{2R},
\]
which is impossible for sufficiently large $R$, since the right-hand side tends to zero.
Repeating this argument along the rows and columns of $Q_R$, we conclude that all lattice nodes in $Q_R$ correspond to the same nonzero integer pair $(s_R,t_R)$. Set $P_R=P_{s_R,t_R}$.

For any $z=x+iy\in Q'_R$, choose the lattice node $z_*=b\sqrt m+ia\sqrt n$ with
$m=\lfloor x^2/b^2\rfloor$ and $n=\lfloor y^2/a^2\rfloor$.
By definition, we have $ 0\le x-b\sqrt m\le b^2/R$ and $0\le y-a\sqrt n\le a^2/R$, which yields $|z-z_*|\le \sqrt{a^4+b^4}/R$.
Thus $z_*\in Q_R$ for large $R$. Since the integer pairs are uniformly
bounded, $P_R$ has Lipschitz constant at most
$\pi N\max\{a^{-2},b^{-2}\}$. Together with \eqref{eq:A:estimate} and \eqref{eq:arithmetic:local}, this gives
\begin{align*}
 |A(z)-P_R(z)|
 &\le |A(z)-A(z_*)|+|A(z_*)-P_R(z_*)|+|P_R(z_*)-P_R(z)|\\
 &\le e_R+\bigl(B_R+\pi N\max\{a^{-2},b^{-2}\}\bigr)
              \frac{\sqrt{a^4+b^4}}R = :E_R,
\end{align*}
uniformly for $z\in Q_R'$, where $E_R/r_R\to0$ as $R\to\infty$.

On the counterclockwise circle $\gamma_R$ of radius $\rho_R=r_R/10$
about $z_R$, since $\gamma_R\subset Q_R'$, Cauchy's theorem gives
$\int_{\gamma_R}A(z)\,\rmd z=0$.
Parametrizing by $z=z_R+\rho_R e^{i\theta}$ gives
$\int_{\gamma_R}x\,\rmd z=i\pi\rho_R^2$ and
$\int_{\gamma_R}y\,\rmd z=-\pi\rho_R^2$. Hence
\[
 \pi^2\rho_R^2\left|\frac{s_R}{b^2}+\frac{t_R}{a^2}\right|
 =\left|\int_{\gamma_R}P_R(z)-A(z)\,\rmd z\right|
 \le 2\pi\rho_R E_R.
\]
Recall that $|s_R|,|t_R|\le N$, the finite set of possible values of $s_R/b^2+t_R/a^2$ has a positive
distance from zero after deleting zero itself. Since $E_R/r_R\to0$,
we have $s_R/b^2+t_R/a^2=0$ for all sufficiently large $R$.
Neither entry of the nonzero pair can vanish, and
$a^2/b^2=-t_R/s_R\in\mathbb Q$.
\end{proof}

For use in the next lemma, we record the local estimates obtained in the preceding proof with a more general bound on $A'$.
Retain the sampling condition and the notation $W,A,Q_R,Q'_R$.
Fix $0<c_1<c_2<\infty$ and $C_0>0$.
Consider disks $D_R=D(z_R,r_R)$, where $z_R=x_R+iy_R$ and $r_R=R^{-1/4}$,
on which $U,W$ are nonvanishing.
Suppose that $c_1R\le x_R,y_R\le c_2R$, and $|A|\le C_0R$, $|A'|\le B_R$ on $D_R$,
where $B_R$ satisfies
\[
\frac{1+B_R}{Rr_R}\longrightarrow0.
\]

These bounds keep the scaled spacings $h,v$ in a fixed positive interval and the parameter $\lambda$ uniformly bounded.
Applying Lemma~\ref{lem:arithmetic:lattice} at each lattice node $z_*\in Q_R$ and returning to the original variables, we obtain integers
$(s_*,t_*)\ne(0,0)$, with $|s_*|,|t_*|\le N$, such that
\[
|A(z_*)-P_{s_*,t_*}(z_*)|\le C\frac{1+B_R}{R}.
\]
Comparing neighboring nodes as above yields a common nonzero integer pair $(s,t)$ for all lattice nodes in $Q_R$, with $|s|,|t|\le N$ for a fixed integer $N$. Approximation by lattice nodes then gives the first estimate below, and integration around the same circle gives the second:
\begin{equation}\label{eq:arithmetic:localdisk}
\sup_{z\in Q'_R}|A(z)-P_{s,t}(z)|
\le C\frac{1+B_R}{R}, \qquad
\left|\frac{s}{b^2}+\frac{t}{a^2}\right|
\le C\frac{1+B_R}{Rr_R}.
\end{equation}
The pair $(s,t)$ may depend on the disk and on $R$, where $N$ and the constants $C$ in these estimates depend only on $a,b,c_1,c_2,C_0$.

The next two lemmas isolate the subharmonic arguments used in the
case where $a^2/b^2$ is rational.

\begin{lemma}\label{lem:arithmetic:compactness}
Let $G\not\equiv0$ be entire and suppose that
\[
 |G(z)|\le C_0e^{T(|z|+1)^2},\qquad z\in\mathbb C,
\]
where $C_0>0$ and $T\ge0$. For every positive sequence $R_n\to\infty$,
the functions $g_n(z)=R_n^{-2}\log|G(R_nz)|$ have a subsequence
converging in $W^{1,1}_{\mathrm{loc}}(\mathbb C)$ to a subharmonic
function $g$ satisfying
\[
 g(z)\le T|z|^2,\qquad g(0)=0.
\]
Here $W^{1,1}_{\mathrm{loc}}$ convergence means local $L^1$ convergence
of the functions and their first weak partial derivatives.

Moreover, if $s_n,s$ are locally integrable subharmonic functions
on an open set $\Omega$ and $s_n\to s$ in $L^1_{\mathrm{loc}}(\Omega)$,
then
\begin{equation}\label{eq:arithmetic:moving-point}
 \limsup_{n\to\infty}s_n(z_n)\le s(z),\quad \forall z_n\to z\in\Omega.
\end{equation}
\end{lemma}
\begin{proof}
The functions $g_n$ are subharmonic and satisfy
\[
 g_n(z)\le T(|z|+R_n^{-1})^2+R_n^{-2}\log C_0.
\]
Therefore $\{g_n\}$ is locally uniformly bounded above. Choose $b$ with
$G(b)\ne0$ and put $b_n=b/R_n$. Since $b_n\to0$ and
$g_n(b_n)=R_n^{-2}\log|G(b)|\to0$, the sequence cannot tend to
$-\infty$ uniformly on compact sets. Proposition~\ref{prop:subharmonic:compactness}
gives a subsequence converging in $L^1_{\mathrm{loc}}$ to a
subharmonic function $g$. The same proposition gives convergence of
the first weak partial derivatives, hence convergence in
$W^{1,1}_{\mathrm{loc}}$.

To prove \eqref{eq:arithmetic:moving-point}, let $z_n\to z\in\Omega$
and fix a real number $c>s(z)$. By upper semicontinuity, we may
choose $\rho>0$ such that $\overline{D(z,2\rho)}\subset\Omega$
and $s\le c$ on this disk. For all sufficiently large $n$,
$D(z_n,\rho)\subset D(z,2\rho)$. The submean inequality gives
\[
 s_n(z_n)
 \le \frac1{\pi\rho^2}\int_{D(z_n,\rho)}s_n\,\rmd A
 \le c+\frac1{\pi\rho^2}
          \int_{D(z,2\rho)}|s_n-s|\,\rmd A.
\]
The last integral tends to zero by local $L^1$ convergence.
Thus $\limsup_n s_n(z_n)\le c$. Since $c>s(z)$ is arbitrary,
\eqref{eq:arithmetic:moving-point} follows.

Finally, for every $z\in\mathbb C$ and $\rho>0$, the submean
inequality, local $L^1$ convergence, and the growth bound for
$g_n$ give
\[
 \begin{aligned}
 g(z)
 &\le \frac1{\pi\rho^2}\int_{D(z,\rho)}g(\zeta)\,\rmd A(\zeta)\\
 &=\lim_{n\to\infty}\frac1{\pi\rho^2}
       \int_{D(z,\rho)}g_n(\zeta)\,\rmd A(\zeta)\\
 &\le \frac{T}{\pi\rho^2}
       \int_{D(z,\rho)}|\zeta|^2\,\rmd A(\zeta)
 =T\left(|z|^2+\frac{\rho^2}{2}\right).
 \end{aligned}
\]
Letting $\rho\downarrow0$ yields $g(z)\le T|z|^2$.
Since $b_n\to0$ and $g_n(b_n)\to0$,
\eqref{eq:arithmetic:moving-point} now gives
\[
 0=\lim_{n\to\infty}g_n(b_n)\le g(0)\le0.
\]
Hence $g(0)=0$.
\end{proof}

\begin{lemma}\label{lem:arithmetic:sector}
Let $v\not\equiv-\infty$ be subharmonic on $\mathbb C$. If
\[
 v(x+iy)\le A+B(|x|-|y|)^2,\qquad x,y\in\mathbb R,
\]
for some $A\in\mathbb R$ and $B\ge0$, then $v$ is constant.
\end{lemma}
\begin{proof}
Fix $0<\delta<\pi/4$, put $\alpha=\pi/4-\delta$, and let
\[
 D_\delta=\{re^{i\theta}:r>0,\ |\theta|<\alpha\}.
\]
The function
\[
 v_\delta(z)=v(z)-A-B\tan\delta\,\operatorname{Re}(z^2)
\]
is subharmonic on $\mathbb C$ and hence upper semicontinuous
on $\overline{D_\delta}$. For $z=x+iy\in\partial D_\delta$,
\[
 (|x|-|y|)^2=\tan\delta\,\operatorname{Re}(z^2),
\]
so the assumed bound on $v$ gives $v_\delta\le0$ on
$\partial D_\delta$. Moreover,
\[
  v_\delta(z)\le B(|x|-|y|)^2-B\tan\delta\,\operatorname{Re}(z^2)
    \le B(1+\tan\delta)|z|^2,\qquad z\in D_\delta.
\]
Choose $2<\rho<\pi/(2\alpha)$. For sufficiently large
$|z|$, $v_\delta(z)\le |z|^\rho$.
Since $D_\delta$ has opening $2\alpha$,
applying Proposition~\ref{prop:subharmonic:phragmen-lindelof}
with $\lambda=\pi/(2\alpha)$ and $M=0$, gives
$v_\delta\le0$ in $D_\delta$.

Letting $\delta\downarrow0$ yields $v\le A$ in $|\arg z|<\pi/4$.
Rotation by multiples of $\pi/2$ gives the same bound in the other
three sectors, and the assumed bound gives it on the diagonal rays.
The subharmonic Liouville theorem now shows that $v$ is constant.
\end{proof}

\begin{lemma}\label{lem:arithmetic:type}
Suppose $a^2=\kappa p$, $b^2=\kappa q$, where $\kappa>0$ and
$p,q$ are coprime positive integers. Let $U,V$ be linearly independent
entire functions such that
\[
 \operatorname{Im}(V(z)\overline{U(z)})=0,
 \qquad z=\pm b\sqrt m\pm ia\sqrt n,\quad m,n\ge1,
\]
and that there are constants $\tau\ge0$ and $C\in\mathbb R$ such that
\[
 \log\sqrt{|U(z)|^2+|V(z)|^2}\le\tau|z|^2+C,\qquad z\in\mathbb C.
\]
Then $\tau\ge\pi/(2\kappa)$. If $\tau=\pi/(2\kappa)$, then
\begin{equation}\label{eq:arithmetic:criticalspan}
 \operatorname{span}_{\mathbb C}\{U,V\}
 =\operatorname{span}_{\mathbb C}
   \{e^{i\pi z^2/(2\kappa)},e^{-i\pi z^2/(2\kappa)}\}.
\end{equation}
\end{lemma}
\begin{proof}
Replacing $U(z),V(z)$ by $U(\sqrt\kappa z),V(\sqrt\kappa z)$
reduces the problem to the case $\kappa=1$, with growth coefficient $\kappa\tau$. We retain the notation $U,V,\tau$ for the rescaled functions and growth coefficient.
When $2\tau>\pi$, the required lower bound already holds.
It therefore suffices to consider $2\tau\le\pi$ and to prove that
equality must hold, together with the asserted description of $U,V$.
Let
\[
 L=\log\sqrt{|U|^2+|V|^2},\qquad W=UV'-U'V\not\equiv0.
\]
Cauchy's estimate on disks of radius one gives
$\log|W(z)|\le 2\tau(|z|+1)^2+C$. In particular $U,V,W$ have finite
quadratic type. For $F=U,V$ set
\[
 A_F=\frac{F'}F-\frac{W'}{2W},\qquad a_{F,R}(z)=\frac{A_F(Rz)}R .
\]

\textit{Step 1.} Fix a compact rectangle $K$ in the open first quadrant and let
$r_R=R^{-1/4}$. Choose $d_K>1$ with $K\subset D(0,d_K-1)$.
The zero-potential estimates in the proof of
Lemma~\ref{lem:arithmetic:disk}, at scale $d_KR$, give a list
$\{\beta_k\}_{k=1}^{N_R}$ of the zeros of $U,V,W$ in an enclosing disk,
counted with multiplicities, with $N_R\lesssim R^2$. On a smaller disk
containing the entire $r_R$-neighborhood of $RK$, these estimates give
\[
 \sum_{G\in\{U,V,W\}}\left|\left(\frac{G'}G\right)'(\zeta)\right|
 \le C+S_R(\zeta),\qquad
 S_R(\zeta)=\sum_{k=1}^{N_R}\frac{1}{|\zeta-\beta_k|^2}
\]
at points away from the zeros.
Set
\[
 Z_R=\bigcup_{\nu=1}^{N_R}\overline{D(\beta_\nu,2r_R)},\qquad
 \Omega_R=RK\setminus Z_R.
\]
As in the proof of Lemma~\ref{lem:arithmetic:disk}, the annular-integral estimates give
\[
 \int_{\Omega_R}S_R\,\rmd A\le CR^2\log R,
\]
for $R\ge2$.
Define
\[
 E_R(K)=K\setminus
 \left\{z\in K:Rz\in\Omega_R,\ S_R(Rz)\le R^{1/4}\right\}.
\]
By Chebyshev's inequality and the change of scale $\zeta=Rz$,
\[
 |E_R(K)|
 \le R^{-2}\left(
 |Z_R|+R^{-1/4}\int_{\Omega_R}S_R\,\rmd A
 \right)
 \le C\left(R^{-1/2}+R^{-1/4}\log R\right).
\]
For $z\in K\setminus E_R(K)$, the same distance comparison
as in the proof of Lemma~\ref{lem:arithmetic:disk} shows that $U,V,W$ are
nonvanishing on $\overline{D(Rz,r_R)}$ and that
\[
 S_R(\zeta)\le 4S_R(Rz)\le 4R^{1/4},
 \qquad |\zeta-Rz|\le r_R.
\]
The preceding logarithmic-derivative estimate therefore gives
\[
 |A_U'(\zeta)|+|A_V'(\zeta)|
 \le B_R:=CR^{1/4},
 \qquad |\zeta-Rz|\le r_R,
\]
where $C$ is a constant independent of $R$.

Now fix $M>0$ and suppose $|a_{F,R}(z)|\le M$ at such a center $z\in K\setminus E_R(K)$.
Then for all sufficiently large $R$, we have
\begin{align*}
    |A_F(\zeta)|&\le|A_F(\zeta)-A_F(Rz)|+|A_F(Rz)| \\
    &\le|\zeta-Rz|\sup_{\overline{D(Rz,r_R)}}|A'_F|+MR\\
    &\le (M+1)R,
\end{align*}
for all $\zeta\in \overline{D(Rz,r_R)}$.
Since $K$ is a fixed compact rectangle in the open first quadrant,
the coordinates of the centre $Rz$ are bounded above and below
by positive constant multiples of $R$, uniformly for $z\in K$.
Moreover, since $B_R=CR^{1/4}$ and $r_R=R^{-1/4}$,
\[
 \frac{1+B_R}{Rr_R}\longrightarrow0, \qquad R\longrightarrow\infty.
\]
Together with the nonvanishing of $U,V,W$ on
$\overline{D(Rz,r_R)}$ and the bounds just established,
this verifies the hypotheses preceding \eqref{eq:arithmetic:localdisk}, with $C_0=M+1$.

Apply \eqref{eq:arithmetic:localdisk} with $a^2=p$ and $b^2=q$, using the pair $(U,V)$
when $F=U$ and the pair $(V,U)$ when $F=V$.
In the latter case, the Wronskian becomes $-W$, which has the
same logarithmic derivative as $W$. Indeed, $(-W)'/(-W)=W'/W$.
The sampling condition is also preserved, since
$\operatorname{Im}(U\overline V)=-\operatorname{Im}(V\overline U)$.
Evaluating the first estimate at the centre $Rz$ and using
the second, we obtain integers $(s,t)\ne(0,0)$, with
$|s|,|t|\le N_{K,M}$, such that
\[
 \left|
 A_F(Rz)-\pi\left(
 \frac{t\operatorname{Im}(Rz)}p
 +i\frac{s\operatorname{Re}(Rz)}q
 \right)
 \right|
 \le C_{K,M}R^{-3/4},
 \qquad
 \left|\frac{s}{q}+\frac{t}{p}\right|
 \le C_{K,M}R^{-1/2}.
\]
The second inequality forces $ps+qt=0$ for sufficiently large $R$.
By coprimality, $s=qk$, $t=-pk$ for a nonzero integer $k$. Consequently
\begin{equation}\label{eq:arithmetic:quantized-scale}
 \left|a_{F,R}(z)-i\pi kz\right|\le C_{K,M}R^{-7/4},
 \qquad 1\le |k|\le N_{K,M}.
\end{equation}
The integer $k$ may depend on $F,R,z$, where
$C_{K,M}$ and $N_{K,M}$ are independent of these choices
for fixed $K,M$.

This estimate holds on a compact rectangle in any open quadrant.
Indeed, for $\omega\in\{1,i,-1,-i\}$ use
$\widetilde U(\zeta)=U(\omega\zeta)$ and
$\widetilde V(\zeta)=V(\omega\zeta)$, interchanging $p,q$ when
$\omega=\pm i$. The transformed logarithmic derivative satisfies
$\widetilde A_F(\zeta)=\omega A_F(\omega\zeta)$. Returning to the
original variable replaces $k$ in \eqref{eq:arithmetic:quantized-scale}
by $\omega^{-2}k=\pm k$. Rotation preserves the area bound for the
exceptional set.

\textit{Step 2.}
Start with an arbitrary sequence $R_n\to\infty$ and write
\[
 u_n(z)=\frac{\log|U(R_nz)|}{R_n^2},\quad
 v_n(z)=\frac{\log|V(R_nz)|}{R_n^2},\quad
 w_n(z)=\frac{\log|W(R_nz)|}{R_n^2}.
\]
Apply Lemma~\ref{lem:arithmetic:compactness} successively to $U,V,W$,
each time extracting from the preceding subsequence. There is an infinite
index set $I$ on which
\[
 u_n\longrightarrow u,\qquad v_n\longrightarrow v,\qquad
 w_n\longrightarrow w \quad\text{in } W^{1,1}_{\mathrm{loc}}(\mathbb C),
\]
with $u(0)=v(0)=w(0)=0$, $u,v\le\tau|z|^2$ and $w\le2\tau|z|^2$.
Here convergence as $n\to\infty$ is understood
with $n\in I$.
Although the limits are initially determined only almost everywhere by the $L^1_{\mathrm{loc}}$ convergence, we choose their unique subharmonic representatives whenever referring to their pointwise values.

Set $\ell_n=R_n^{-2}L(R_n\,\cdot)$ and $\ell=\max\{u,v\}$, which is
subharmonic. Since
\begin{align*}
 0\le\ell_n-\max\{u_n,v_n\}&\le\frac{\log 2}{2R_n^2},\\
 |\max\{u_n,v_n\}-\max\{u,v\}|&\le|u_n-u|+|v_n-v|,
\end{align*}
almost everywhere,
we have $\ell_n\to\ell$ in $L^1_{\mathrm{loc}}$.
To relate the Wronskian limit to $\ell$, fix $z\in\mathbb C$
and $\rho>0$, and let
\[
 M_n:=\exp\left(
 R_n^2\sup_{\overline{D(z,2\rho)}}\ell_n
 \right).
\]
By the definition of $\ell_n$, we have
\[
 |U(R_n\eta)|^2+|V(R_n\eta)|^2
 =e^{2R_n^2\ell_n(\eta)}
 \le M_n^2,
 \qquad \eta\in\overline{D(z,2\rho)}.
\]
For every $\zeta\in D(z,\rho)$,
$\overline{D(\zeta,\rho)}\subset D(z,2\rho)$.
Cauchy's formula applied to $\eta\mapsto U(R_n\eta)$ gives
\[
 R_nU'(R_n\zeta)
 =\frac{1}{2\pi i}
   \int_{|\eta-\zeta|=\rho}
   \frac{U(R_n\eta)}{(\eta-\zeta)^2}\,\rmd\eta.
\]
Taking absolute values, and arguing in the same way for $V$,
we obtain
\[
 |U'(R_n\zeta)|\le\frac{M_n}{R_n\rho},
 \qquad
 |V'(R_n\zeta)|\le\frac{M_n}{R_n\rho}.
\]
Since $W=UV'-U'V$, it follows that
\[
 \begin{aligned}
 |W(R_n\zeta)|
 \le |U(R_n\zeta)|\,|V'(R_n\zeta)|
      +|U'(R_n\zeta)|\,|V(R_n\zeta)|
 \le\frac{2M_n^2}{R_n\rho}.
 \end{aligned}
\]
Taking logarithms and dividing by $R_n^2$ yields
\[
 w_n(\zeta)
 \le 2\sup_{\overline{D(z,2\rho)}}\ell_n
       +\frac{1}{R_n^2}\log\frac{2}{R_n\rho},
 \qquad \zeta\in D(z,\rho).
\]

Since $\ell_n$ is subharmonic, \eqref{eq:arithmetic:moving-point}
at a subsequence of approximate maximizers gives
\[
 \limsup_{n\in I}\sup_{\overline{D(z,2\rho)}}\ell_n
 \le\sup_{\overline{D(z,2\rho)}}\ell.
\]
The upper bound for $w_n$ passes to $w$ almost everywhere on
$D(z,\rho)$ and then to $w(z)$ by the submean inequality. Letting
$\rho\downarrow0$ and using upper semicontinuity gives $w(z)\le2\ell(z)$.
Define, in the almost-everywhere sense,
\[
 H_U=u-\frac w2,\qquad H_V=v-\frac w2,\qquad
 H=\max\{H_U,H_V\}=\ell-\frac w2 .
\]
The functions $H_U,H_V$ belong to $W^{1,1}_{\mathrm{loc}}$. We have
\begin{equation*}
 H\ge0\quad\mbox{ a.e. },\qquad
 \ell(z)\le\tau|z|^2,\qquad \ell(0)=w(0)=0.
\end{equation*}
With $2\partial_z=\partial_x-i\partial_y$, the identity
$2\partial_z\log|F|=F'/F$ holds away from the zeros and also
identifies the weak derivative, since logarithmic singularities
have locally integrable gradients. The $W^{1,1}_{\mathrm{loc}}$
convergence gives, along $I$,
\[
 a_{F,R_n}\longrightarrow a_F:=2\partial_zH_F
 \quad\text{in }L^1_{\mathrm{loc}},\qquad F=U,V.
\]

\textit{Step 3.}
Cover the four open quadrants by countably many compact rectangles
$K_m$, each contained in one quadrant. For positive integers $m,M$, let
$E_R^{(m,M)}\subset K_m$ be the exceptional set in
\eqref{eq:arithmetic:quantized-scale}. Choose increasing indices $n_j\in I$ recursively so that
\[
 |E_{R_{n_j}}^{(m,M)}|\le 2^{-j},
 \qquad 1\le m,M\le j,
\]
and
\[
 \int_{K_m}
 \bigl(|a_{U,R_{n_j}}-a_U|+|a_{V,R_{n_j}}-a_V|\bigr)
 \,\rmd A\le 2^{-j},
 \qquad 1\le m\le j.
\]
This is possible because, at each step, only finitely many
conditions are imposed, and all the quantities on the left
tend to zero as $n\to\infty$ along $I$.

For each fixed $m,M$, Tonelli's theorem gives
\[
 \sum_{j\ge\max\{m,M\}}
 \mathbf{1}_{E_{R_{n_j}}^{(m,M)}}(z)<\infty
 \quad  \mbox{ a.e. }  z\in K_m,
\]
where $\mathbf{1}_E$ denotes the indicator function of $E$.
Thus almost every point of $K_m$ belongs to only finitely
many of these exceptional sets. Similarly, for each fixed $m$,
\[
 \sum_{j\ge m}
 \bigl(
 |a_{U,R_{n_j}}(z)-a_U(z)|
 +|a_{V,R_{n_j}}(z)-a_V(z)|
 \bigr)<\infty
 \quad \mbox{ a.e. } z\in K_m.
\]
In particular, both errors tend to zero.

Since there are only countably many indices $m,M$, we may
exclude a single set $N$ of planar measure zero such that,
for every $m$ and every $z\in K_m\setminus N$,
\[
 a_{F,R_{n_j}}(z)\longrightarrow a_F(z)\in\mathbb C,
 \qquad F=U,V.
\]
Moreover, for every positive integer $M$, there is an index
$J=J(z,m,M)$ such that
\[
 z\notin E_{R_{n_j}}^{(m,M)},
 \qquad \forall j\ge J.
\]

Fix $m$ and such a point $z\in K_m\setminus N$. Let $F\in\{U,V\}$.
Choose a positive integer $M>|a_F(z)|+1$.
For large $j$, we have
$z\notin E_{R_{n_j}}^{(m,M)}$ and $|a_{F,R_{n_j}}(z)|\le M$.
By \eqref{eq:arithmetic:quantized-scale},
$a_{F,R_{n_j}}(z)-i\pi k_jz\to0$ for nonzero integers $k_j$.
Since $a_{F,R_{n_j}}(z)\to a_F(z)$, the preceding estimate
implies $i\pi k_jz\to a_F(z)$. Hence, for all sufficiently
large $j,\ell$,
\[
 \pi|z|\,|k_j-k_\ell|
 \le |i\pi k_jz-a_F(z)|+|i\pi k_\ell z-a_F(z)|
 <\pi|z|.
\]
Since $z\ne0$, we obtain $|k_j-k_\ell|<1$.
The numbers $k_j$ are integers, so $k_j=k_\ell$.
Thus $k_j$ is eventually equal to a fixed nonzero integer,
which we denote by $k_F(z)$.
Hence, for $F=U,V$
and almost every $z$ off the axes,
\begin{equation}\label{eq:arithmetic:individualgradients}
 2\partial_z H_F(z)=i\pi k_F(z)z,\qquad
 k_F(z)\in\mathbb Z\setminus\{0\}.
\end{equation}

By the Sobolev chain rule,
\[
 H=\max\{H_U,H_V\}=H_V+(H_U-H_V)^+
 \in W^{1,1}_{\mathrm{loc}}(\mathbb C),
\]
and its weak derivative satisfies
\[
 2\partial_zH=
 \begin{cases}
  2\partial_zH_U,& H_U>H_V,\\
  2\partial_zH_V,& H_U\le H_V,
 \end{cases}
 \qquad\mbox{ a.e. }
\]
On $\{H_U=H_V\}$, the weak gradients of $H_U$ and $H_V$
agree $\mbox{ a.e. }$, so either choice gives the same value
there. Define $k(z)=k_U(z)$ when $H_U(z)>H_V(z)$ and
$k(z)=k_V(z)$ otherwise. Then
\eqref{eq:arithmetic:individualgradients} gives
\[
 2\partial_zH(z)=i\pi k(z)z,\qquad
 k(z)\in\mathbb Z\setminus\{0\}.
\]
for almost every $z$ off the axes.

\textit{Step 4.}
In the first quadrant, set $z=x+iy$ and
$\zeta=z^2/2=\xi+i\eta$, so that
$\xi=(x^2-y^2)/2$ and $\eta=xy$.
This change of variables maps the first quadrant conformally onto
the upper half-plane. We take the square root with positive real
and imaginary parts and put $h(\zeta)=H(\sqrt{2\zeta})$.
Then $h\in W^{1,1}_{\mathrm{loc}}$ in the upper half-plane, and
the chain rule gives
\[
 h_\xi=\frac{xH_x-yH_y}{x^2+y^2}=0,\qquad
 h_\eta=\frac{yH_x+xH_y}{x^2+y^2}=-\pi k(z)
 \quad\mbox{ a.e. }
\]
We use the absolute continuity on lines property for functions
in $W^{1,1}_{\mathrm{loc}}$: on almost every line parallel to a coordinate axis,
the restriction has a locally absolutely continuous representative,
whose derivative agrees with the corresponding weak partial
derivative; see \cite[Theorem~4.21]{EvansGariepy2015}.
For almost every $\eta>0$, the horizontal representative satisfies
\[
 h(\xi_2+i\eta)-h(\xi_1+i\eta)
 =\int_{\xi_1}^{\xi_2}h_\xi(s+i\eta)\,\rmd s=0,
 \qquad \xi_1<\xi_2.
\]
Thus $h$ is independent of $\xi$ almost everywhere. By Fubini's
theorem and the same property on vertical lines, we may choose
$\xi_0$ such that the vertical section has a locally absolutely
continuous representative $\Phi(\eta)=h(\xi_0+i\eta)$ on
$(0,\infty)$, with $h(\xi+i\eta)=\Phi(\eta)$ almost everywhere.
Taking weak derivatives gives
\[
h_\eta(\xi+i\eta)=\Phi'(\eta)
\quad\mbox{ a.e. }(\xi,\eta)\in\mathbb R\times(0,\infty).
\]
The formula for $h_\eta$ and Fubini's theorem therefore imply
\begin{equation}\label{eq: Phi pi}
 \Phi'(\eta)\in\pi(\mathbb Z\setminus\{0\})
 \quad\mbox{ a.e. }\eta>0.
\end{equation}
Returning to the original coordinates, we obtain
\[
 H(x+iy)=\Phi(xy)\quad\mbox{ a.e. }\text{in the first quadrant}.
\]
Since $H\ge0$ almost everywhere, $\Phi\ge0$ almost everywhere
on $(0,\infty)$. Continuity of $\Phi$ then gives
$\Phi(t)\ge0$ for every $t>0$.

We next estimate the distributional Laplacian of $H$. Let $v$
be subharmonic on $\mathbb C$, with $v(0)=0$ and
$v(z)\le A|z|^2$, and write $\mu_v=\Delta v$, a nonnegative
measure. Proposition~\ref{prop:subharmonic:jensen}, with radius
$2s$, gives
\[
 \frac1{2\pi}\int_{|z|<2s}\log\frac{2s}{|z|}\,\rmd\mu_v(z)
 =\frac1{2\pi}\int_0^{2\pi}v(2se^{it})\,\rmd t-v(0)
 \le4As^2.
\]
The logarithmic kernel is nonnegative on $D(0,2s)$ and is at
least $\log2$ on $D(0,s)$. Since $\mu_v$ is nonnegative,
\[
 \begin{aligned}
 (\log2)\mu_v(D(0,s))
 &\le\int_{|z|<s}\log\frac{2s}{|z|}\,\rmd\mu_v(z)\\
 &\le\int_{|z|<2s}\log\frac{2s}{|z|}\,\rmd\mu_v(z)
 \le8\pi As^2.
 \end{aligned}
\]
Apply this with $(v,A)=(\ell,\tau)$ and $(v,A)=(w,2\tau)$.
Since $H=\ell-w/2$, its distributional Laplacian is a signed
Radon measure, and its total variation satisfies
\begin{equation}\label{eq:arithmetic:variation}
 \begin{aligned}
 |\Delta H|(D(0,s))
 &\le\Delta\ell(D(0,s))+\tfrac12\Delta w(D(0,s))\\
 &\le \frac{16\pi\tau}{\log2} s^2,\qquad s>0.
 \end{aligned}
\end{equation}

To estimate the variation of $\Phi'$, fix a nonnegative
$\chi\in C_c^\infty(0,1)$ with $\int_{\mathbb R}\chi=1$.
For $\psi\in C_c^\infty(0,\infty)$, choose $T>0$ such that
$\operatorname{supp}\psi\subset(0,T)$, and set
\[
 \varphi_T(\xi,\eta)=\chi(\xi/T)\psi(\eta),\qquad
 \Theta_T(x+iy)=\varphi_T\bigl((x^2-y^2)/2,xy\bigr),
 \quad x,y>0.
\]
Extend $\Theta_T$ by zero outside the first quadrant. Its support
stays away from the axes, so this extension is smooth. On its
support, $0<\xi,\eta<T$, and hence
\[
 x^2+y^2=2\sqrt{\xi^2+\eta^2}<4T.
\]
Thus $\Theta_T\in C_c^\infty(D(0,2\sqrt T))$ and
$\|\Theta_T\|_\infty\le\|\chi\|_\infty\|\psi\|_\infty$.
Under the above change of variables,
\[
 \Delta_{x,y}\Theta_T
 =(x^2+y^2)(\Delta_{\xi,\eta}\varphi_T)(\xi,\eta),
 \qquad \rmd\xi\,\rmd\eta=(x^2+y^2)\,\rmd x\,\rmd y,
\]
Here $\Delta_{x,y}=\partial_x^2+\partial_y^2$ and
$\Delta_{\xi,\eta}=\partial_\xi^2+\partial_\eta^2$.
Consequently, using $H(x+iy)=\Phi(xy)$ almost everywhere in the first quadrant yields
\[
 \begin{aligned}
 \langle\Delta H,\Theta_T\rangle
 &=\int_{\mathbb C}H\Delta\Theta_T\,\rmd A\\
 &=\int_0^\infty\int_{\mathbb R}\Phi(\eta)
   \left(T^{-2}\chi''(\xi/T)\psi(\eta)
         +\chi(\xi/T)\psi''(\eta)\right)\,\rmd\xi\,\rmd\eta\\
 &=T\int_0^\infty\Phi(\eta)\psi''(\eta)\,\rmd\eta,
 \end{aligned}
\]
where the last equality uses $\int\chi''=0$ and $\int\chi=1$.
On the other hand, \eqref{eq:arithmetic:variation} gives
\[
 |\langle\Delta H,\Theta_T\rangle|
 \le |\Delta H|(D(0,2\sqrt T))\|\Theta_T\|_\infty
 \le\frac{64\pi\tau}{\log2}T\|\chi\|_\infty\|\psi\|_\infty.
\]
Dividing by $T$ and integrating by parts, we obtain
\[
 \left|\int_0^\infty\Phi'(t)\psi'(t)\,\rmd t\right|
 =\left|\int_0^\infty\Phi(t)\psi''(t)\,\rmd t\right|
 \le\frac{64\pi\tau}{\log2}\|\chi\|_\infty\|\psi\|_\infty.
\]
The bound is independent of $\psi$ and of its support.
Since
\[
 \langle\Phi'',\psi\rangle
 =-\int_0^\infty\Phi'(t)\psi'(t)\,\rmd t,
\]
the Riesz representation theorem shows that the distribution
$\Phi''$ is a finite signed measure on $(0,\infty)$, with
\[
 |\Phi''|((0,\infty))\le\frac{64\pi\tau}{\log2}\|\chi\|_\infty.
\]
Here $|\Phi''|$ denotes its total variation measure.
Choose $t_0>0$ such that
$|\Phi''|((t_0,\infty))<\pi$.

Let $t_0<s<t$ be Lebesgue points of $\Phi'$ at which
\eqref{eq: Phi pi} holds, and choose $t_1>t$.
Since $\Phi'\in L^1(t_0,t_1)$,
\cite[Theorem~5.21]{EvansGariepy2015} gives
\[
 |\Phi'(t)-\Phi'(s)|
 \le|\Phi''|((t_0,t_1))
 \le|\Phi''|((t_0,\infty))<\pi.
\]
The difference is an integer multiple of $\pi$, so
$\Phi'(t)=\Phi'(s)$. Since such points form a set of full
measure in $(t_0,\infty)$, there is a fixed integer $k_0\ne0$
such that
\[
 \Phi'(t)=\pi k_0
 \quad\mbox{ a.e. }t>t_0.
\]
Absolute
continuity of $\Phi$ gives
\[
 \Phi(t)=\Phi(t_0)+\pi k_0(t-t_0),\qquad t\ge t_0.
\]
Its nonnegativity forces $k_0\ge1$, so
$\Phi(t)\ge\pi t-\pi t_0$ for $t\ge t_0$.
For $0<t<t_0$, the same bound follows from $\Phi(t)\ge0$.

Rotation by multiples of $\pi/2$ preserves the gradient condition
(up to the sign of $k$) and the measure bound
\eqref{eq:arithmetic:variation}. The same argument therefore
applies in each quadrant, with a possibly different function
$\Phi$ and constant. Let $C_0$ be the largest of these four
constants. Then, in each quadrant,
\begin{equation}\label{eq:arithmetic:quadrant-Phi}
 H(x+iy)=\Phi(|xy|)\quad\mbox{ a.e. },\qquad
 \Phi(t)\ge\pi t-C_0\quad(t>0).
\end{equation}
Since $w=2\ell-2H$, $\ell(z)\le\tau|z|^2$, and $2\tau\le\pi$,
we obtain
\[
 w(x+iy)\le 2C_0+2\tau(x^2+y^2)-2\pi|xy|
 \le 2C_0+\pi(|x|-|y|)^2\quad\mbox{ a.e. }
\]
Averaging this inequality over a disk about any point, using
the submean inequality for $w$, and letting the radius tend to
zero extends the bound to every point, since the majorant is
continuous.

Applying Lemma~\ref{lem:arithmetic:sector} to $w$ and using
$w(0)=0$, we get $w\equiv0$ and $H=\ell$ almost everywhere.
In each quadrant, \eqref{eq:arithmetic:quadrant-Phi} and
$H=\ell$ therefore give $\Phi(|xy|)\le\tau(x^2+y^2)$ almost
everywhere and hence everywhere, by continuity of both sides. Setting
$|x|=|y|=\sqrt t$ gives
\[
 \pi t-C_0\le\Phi(t)\le2\tau t,\qquad t>0.
\]
Dividing by $t$ and letting $t\to\infty$ yields $2\tau\ge\pi$.
Together with $2\tau\le\pi$, this gives $2\tau=\pi$ and
$\Phi(t)\le\pi t$. Thus $\ell(x+iy)\le\pi|xy|$ almost
everywhere and then everywhere by the submean inequality.
Since the initial sequence $R_n\to\infty$ was arbitrary,
the above choice of $I$ and then $n_j\in I$ can be made
for any such sequence. The resulting limit $\ell$ always
satisfies $\ell(x+iy)\le\pi|xy|$.

Fix $\varepsilon>0$. Suppose that
\[
 L(x+iy)\le\pi|xy|+\varepsilon(x^2+y^2)
\]
fails at points $z_j=R_j\omega_j$, where $R_j\to\infty$
and $|\omega_j|=1$. By the preceding argument, after passing
to a subsequence,
$\omega_j\to\omega$ and
$R_j^{-2}L(R_j\,\cdot)\to\ell$ in $L^1_{\mathrm{loc}}$,
with $\ell(x+iy)\le\pi|xy|$.
The moving-point estimate
\eqref{eq:arithmetic:moving-point} then gives
\[
 \varepsilon
 \le\limsup_j\left(
 \frac{L(z_j)}{R_j^2}
 -\pi|\operatorname{Re}\omega_j\operatorname{Im}\omega_j|
 \right)
 \le\ell(\omega)
 -\pi|\operatorname{Re}\omega\operatorname{Im}\omega|
 \le0,
\]
which is a contradiction. Thus the asserted bound for $L$ holds
outside a disk. Since $L$ is bounded above on that disk
and $|U|+|V|\le\sqrt2\,e^L$, we obtain
\begin{equation}\label{eq:arithmetic:endpointgrowth}
 |U(x+iy)|+|V(x+iy)|
 \le C_\varepsilon e^{\pi|xy|+\varepsilon(x^2+y^2)},
 \qquad x,y\in\mathbb R.
\end{equation}

\textit{Step 5.}
Write $Q=V/U$ and $g^*(z)=\overline{g(\bar z)}$ for entire or
meromorphic $g$. For $d=1,2$, choose integers $m_0,n_0$ with
$qm_0-pn_0=d$. For all sufficiently large integers $k$, put
\[
 m_k=m_0+pk,\qquad n_k=n_0+qk,\qquad
 Y_k=2\sqrt{pq\,m_kn_k}.
\]
The numbers $m_k,n_k$ are then positive, and $Y_k$ is strictly
increasing. There are a constant $C_Y>0$ and an integer $K_0\ge1$
such that $|Y_k-2pqk|\le C_Y$ for every $k\ge K_0$.
In $\operatorname{Re}s>0$ define
\[
 F_d(s)=V(\sqrt{d+is})U^*(\sqrt{d-is})
        -U(\sqrt{d+is})V^*(\sqrt{d-is})
\]
using the principal square roots. This is holomorphic since the two
root arguments have nonzero imaginary parts. At $s=Y_k$ the roots
are $\sqrt{qm_k}+i\sqrt{pn_k}$ and its conjugate, so the sampled
reality condition gives $F_d(Y_k)=0$.
For either root $z$,
$|\operatorname{Re}z\,\operatorname{Im}z|=\operatorname{Re}s/2$
and $|z|^2\le d+|s|$. Applying \eqref{eq:arithmetic:endpointgrowth}
with $\varepsilon/2$ to each factor and absorbing $e^{\varepsilon d}$
into the constant gives, for every $\varepsilon>0$,
\[
 |F_d(s)|\le C_\varepsilon e^{\pi\operatorname{Re}s+\varepsilon|s|},
 \qquad \operatorname{Re}s>0.
\]
Set $f_d(s)=F_d(1+2pq s)$. Since $F_d$ is holomorphic in
$\operatorname{Re}s>0$, the function $f_d$ is holomorphic in
$\operatorname{Re}s>-1/(2pq)$.
For every $\varepsilon>0$, applying the preceding estimate
with $\varepsilon/(2pq)$ and absorbing the constant factors gives
\[
 |f_d(s)|
 \le C_\varepsilon e^{2\pi pq\operatorname{Re}s+\varepsilon|s|}
 \le C_\varepsilon e^{(2\pi pq+\varepsilon)|s|},
 \qquad \operatorname{Re}s\ge0.
\]
Taking $\varepsilon=1$ shows that $f_d$ is at most of normal
type with respect to $\varrho(r)\equiv1$ in the right
half-plane. On the imaginary axis,
$|f_d(it)|\le C_\varepsilon e^{\varepsilon|t|}$,  $t\in\mathbb R$.
Hence
\[
 H_{f_d}^{[1]}(\pm\pi/2)
 =\limsup_{r\to\infty}\frac{\log|f_d(\pm ir)|}{r}
 \le\varepsilon.
\]
Since $\varepsilon>0$ is arbitrary,
$H_{f_d}^{[1]}(\pm\pi/2)\le0$.
Its positive zeros
$\lambda_k=(Y_k-1)/(2pq)$ satisfy $\lambda_k/k\to1$ as $k\to\infty$.
Applying Lemma~\ref{lem:levin-ray} with $m=1$ and $a=1$,
 we have $f_d\equiv0$.
Therefore $F_d\equiv0$, $d=1,2$.

For $d=1,2$, set
\[
 z(t)=\sqrt{d/2}\,(t+t^{-1}),\qquad
 w(t)=-i\sqrt{d/2}\,(t-t^{-1}),
 \qquad t\in\mathbb C\setminus\{0\}.
\]
These parametrize all pairs satisfying $z^2+w^2=2d$,
with inverse $t=(z+iw)/\sqrt{2d}$. The function
\[
 G_d(t)=V(z(t))U^*(w(t))-U(z(t))V^*(w(t))
\]
is holomorphic on $\mathbb C\setminus\{0\}$.
For $s>0$, put
\[
 t_s=\frac{\sqrt{d+is}+i\sqrt{d-is}}{\sqrt{2d}}.
\]
Then $G_d(t_s)=F_d(s)=0$. Since $z(t_s)^2=d+is$,
the zeros $t_{1+1/n}$ are distinct and converge to $t_1\ne0$.
The identity theorem gives $G_d\equiv0$, and division yields
the meromorphic identity
\[
 Q(z)=Q^*(w),\qquad z^2+w^2=2d.
\]

Comparing $(z,w)$ and $(z,-w)$ gives
$Q^*(w)=Q^*(-w)$. Since every $w\in\mathbb C$ occurs
in such a pair, $Q^*$ and hence $Q$ are even.
Thus $Q(z)=r(z^2)$ for a meromorphic function $r$ on
$\mathbb C$.
The preceding identities become
\[
 r(\zeta)=r^*(2d-\zeta),\qquad d=1,2.
\]
Conjugate reflection of the identity for $d=1$ gives
$r^*(\eta)=r(2-\eta)$. Combining this with the identity
for $d=2$ yields
\[
 r(\zeta)=r^*(4-\zeta)=r(\zeta-2).
\]
Thus $r$ is $2$-periodic.

\textit{Step 6.}
Since $r$ is nonconstant, choose $\alpha_1,\alpha_2\in(0,2)$
such that $c_j=r(\alpha_j)$ are finite and distinct. Put
\[
 P_j(z)=V(z)-c_jU(z),\qquad
 E_j(z)=\sin\!\left(\frac{\pi}{2}(z^2-\alpha_j)\right)
 \qquad j=1,2.
\]
The zeros of $E_j$ satisfy $z^2=\alpha_j+2n$, $n\in\mathbb Z$.
They are nonzero and simple, since $0<\alpha_j<2$ and
$E_j'(z)=\pi z\cos(\pi(z^2-\alpha_j)/2)$.
If $z_0$ is any such zero, periodicity gives $Q(z_0)=c_j$ in the
meromorphic sense, with a finite value. Hence $Q-c_j$ is holomorphic
near $z_0$ and vanishes there. The meromorphic identity
$P_j=U(Q-c_j)$ then shows that $P_j(z_0)=0$, including when
$U(z_0)=V(z_0)=0$. Every possible singularity of $P_j/E_j$ is therefore removable.
Thus the quotient is entire. It is not identically zero since $U$  and $V$ are linearly independent.

Fix $\sigma=1/4$. There exists $c_\sigma>0$ such that
\[
 |\sin(\pi w/2)|
 \ge c_\sigma e^{(\pi/2)|\operatorname{Im}w|}
\]
whenever $w$ is at distance at least $\sigma$ from $2\mathbb Z$.
Combining this bound with the original growth estimate at
$\tau=\pi/2$ gives
\[
 \left|\frac{P_j(z)}{E_j(z)}\right|\le C_j e^{\Upsilon(z)},\qquad
 \Upsilon(z)=\frac{\pi}{2}\bigl(|z|^2-|\operatorname{Im}(z^2)|\bigr)
     =\frac{\pi}{2}(|x|-|y|)^2,
\]
outside the bounded open sets
$\Omega_{j,n}=\{z:|z^2-\alpha_j-2n|<\sigma\}$.
If $z\in\Omega_{j,n}$ and $w\in\partial\Omega_{j,n}$, then
$|w^2-z^2|\le2\sigma$. Since
$\Upsilon(z)=\frac\pi2(|z^2|-|\operatorname{Im}(z^2)|)$, the triangle
inequality gives
\[
 |\Upsilon(w)-\Upsilon(z)|\le\pi|w^2-z^2|\le2\pi\sigma.
\]
For $w\in\partial\Omega_{j,n}$, we have
$|w^2-\alpha_j-2n|=\sigma$ and
$|w^2-\alpha_j-2k|\ge2-\sigma>\sigma$ whenever $k\ne n$.
Thus the quotient bound holds on $\partial\Omega_{j,n}$.
For every $z\in\Omega_{j,n}$, the maximum modulus principle
and the preceding estimate give
\[
 \left|\frac{P_j(z)}{E_j(z)}\right|
 \le \max_{w\in\partial\Omega_{j,n}}
       \left|\frac{P_j(w)}{E_j(w)}\right|
 \le C_j e^{\Upsilon(z)+2\pi\sigma}.
\]
Absorbing $e^{2\pi\sigma}$ into $C_j$, we obtain
$|P_j(z)/E_j(z)|\le C_j e^{\Upsilon(z)}$ ,
$\forall z\in\mathbb C$.
Lemma~\ref{lem:arithmetic:sector} applied to $\log|P_j/E_j|$
shows that this logarithm is constant. Hence the entire quotient
has constant modulus and is itself constant:
$P_j=d_jE_j$ for some constant $d_j\ne0$.

Both $P_1$ and $P_2$ belong to
$\operatorname{span}_{\mathbb C}\{e^{i\pi z^2/2},e^{-i\pi z^2/2}\}$.
Since $c_1\ne c_2$,
\[
 U=\frac{P_1-P_2}{c_2-c_1},\qquad V=P_1+c_1U.
\]
Thus $U,V$ belong to the same span. Their linear independence
gives equality of the two spans.
Rescaling gives \eqref{eq:arithmetic:criticalspan}.
\end{proof}

\begin{lemma}\label{lem:arithmetic:endpoint}
Let $p,q$ be coprime positive integers and $0<\kappa\le1$.
There are no linearly independent $U,V\in F^2(\mathbb C)$ satisfying
\[
 \operatorname{Im}(V(z)\overline{U(z)})=0,
 \qquad z=\pm\sqrt{\kappa qm}\pm i\sqrt{\kappa pn},\quad m,n\ge1.
\]
\end{lemma}
\begin{proof}
The Fock evaluation estimate gives the growth bound with $\tau=\pi/2$.
If $\kappa<1$, Lemma~\ref{lem:arithmetic:type} is already a
contradiction. If $\kappa=1$, its equality classification would put
$e^{i\pi z^2/2}$ in $F^2(\mathbb C)$, whereas
\[
 |e^{i\pi z^2/2}|^2e^{-\pi|z|^2}=e^{-\pi(x+y)^2}
\]
has infinite integral over the plane.
\end{proof}

\begin{proof}[Proof of Theorem~\ref{thm:t1}]
If $a<1$, take the three horizontal lines at heights $b,b\sqrt2,2b$.
The distances from the first line to the other two have ratio
$\sqrt2-1\notin\mathbb Q$. Thus Theorem~\ref{thm:t2b} gives Gabor
phase retrieval on $a\mathbb Z^{1/2}\times\{b,b\sqrt2,2b\}$. Hence so does $\Lambda$. The case $b<1$ follows
similarly.

For the remaining positive assertions, suppose that $\Lambda$
fails Gabor phase retrieval. Lemma~\ref{Lm:L5} gives independent
$U,V\in F^2(\mathbb C)$ with real sampled product on
$\Gamma^*_\Lambda=\{\pm b\sqrt m\pm ia\sqrt n:m,n\ge0\}$.
The Fock evaluation estimate gives
\[
 \log\sqrt{|U(z)|^2+|V(z)|^2}
 \le\frac{\pi}{2}|z|^2+C.
\]
If $a^2/b^2$ is irrational,
Lemma~\ref{lem:arithmetic:rigidity} gives a contradiction.
If $a^2=\kappa p$, $b^2=\kappa q$ with $0<\kappa\le1$ and
$p,q$ coprime, Lemma~\ref{lem:arithmetic:endpoint} excludes the
independent Fock pair. This proves the positive assertions.

For $\kappa>1$, set $c=\pi/\kappa\in(0,\pi)$. Then
$a^2c=\pi p$ and $b^2c=\pi q$, therefore Lemma~\ref{Lm:L6}
gives nonuniqueness.
\end{proof}

\begin{proof}[Proof of Corollary~\ref{coro1.2}]
In the rational case, $\kappa\le\min\{a^2,b^2\}$.
If $a=b$, then $p=q=1$ and $\kappa=a^2$.
Both assertions therefore follow from Theorem~\ref{thm:t1}.
\end{proof}

\section{Power lattices between the square-root and linear scales}\label{sec:power}

We use Lemmas~\ref{lem:arithmetic:disk} and
\ref{lem:arithmetic:lattice} to prove the case $1/2<\nu<1$ of
Theorem~\ref{thm:power}.
For $a,b>0$ and $1/2<\nu<1$, put
\[
 \Lambda^+_{a,b,\nu}=\{am^\nu+ibn^\nu:m,n\ge1\}.
\]
An entire function has finite quadratic type if it satisfies
$|f(z)|\le Ce^{T|z|^2}$ for some $C,T>0$.

To consider general power lattices,  we need the following
two lemmas.
\begin{lemma}\label{lem:power:real}
Let $U,V$ be entire functions of finite quadratic type. If
\[
 \operatorname{Im}\bigl(V(\lambda)\overline{U(\lambda)}\bigr)=0
 \qquad(\lambda\in\Lambda^+_{a,b,\nu}),
\]
then $U,V$ are linearly dependent over $\mathbb C$.
\end{lemma}

\begin{lemma}\label{lem:power:phase}
Let $F,G$ be entire functions of finite quadratic type. If
$|F(\lambda)|=|G(\lambda)|$ on $\Lambda^+_{a,b,\nu}$, then
$G=cF$ for some constant $c$ with $|c|=1$. In particular,
this conclusion holds for $F,G\in F^2(\mathbb C)$.
\end{lemma}

We first record estimates for the proof of Lemma~\ref{lem:power:real}.
Constants may depend on the fixed functions and on $a,b,\nu$,
but not on $R\ge2$. Set
\[
 \beta=\nu^{-1}\in(1,2),\qquad \alpha=\beta-1,\qquad
 \epsilon=\alpha/10,\qquad h=R^{-\alpha},\qquad
 r=\tfrac12R^{-\epsilon}.
\]
Write $x_m=am^\nu$ and $y_n=bn^\nu$. For $X(t)=at^\nu$,
with $t>0$, we have
\[
 X'(t)=\nu a^\beta X(t)^{1-\beta},\qquad
 |X''(t)|=\nu(1-\nu)a^{2\beta}X(t)^{1-2\beta}.
\]
Set
\[
 \begin{aligned}
 d_-&=\nu2^{1-\beta}\min\{a^\beta,b^\beta\},
 &d_+&=\nu\max\{a^\beta,b^\beta\},\\
 J_0&=\nu(1-\nu)\max\{a^{2\beta},b^{2\beta}\}.
 \end{aligned}
\]
Whenever $R\le X(t)\le2R$, these choices give
\[
 d_-h\le X'(t)\le d_+h,\qquad |X''(t)|\le J_0h^2/R.
\]
The same bounds hold for $Y(t)=bt^\nu$.
Integrating the first derivatives and using
\[
 X(m+2)-2X(m+1)+X(m)
 =\int_0^1\int_0^1X''(m+s+t)\,\rmd s\,\rmd t
\]
gives the following estimates whenever the coordinates involved
lie in $[R,2R]$:
\begin{align}
 d_-h\le x_{m+1}-x_m\le d_+h,&\qquad
 d_-h\le y_{n+1}-y_n\le d_+h,
 \notag\\
 |x_{m+2}-2x_{m+1}+x_m|\le J_0h^2/R,&\qquad
 |y_{n+2}-2y_{n+1}+y_n|\le J_0h^2/R.
 \label{eq:power:spacing2}
\end{align}

For $x,y\in[R,2R]$, choose
$m=\lceil(x/a)^\beta\rceil$ and $n=\lceil(y/b)^\beta\rceil$.
Since $X,Y$ are increasing and concave,
\[
 0\le x_m-x\le\nu a^\beta x^{1-\beta}\le d_+h,\qquad
 0\le y_n-y\le\nu b^\beta y^{1-\beta}\le d_+h.
\]
Thus
\begin{equation}\label{eq:power:node-distance}
 |x_m+iy_n-(x+iy)|\le\sqrt2d_+h.
\end{equation}
If $x,y\le2R-d_+h$, the selected node also lies in $[R,2R]^2$.

\begin{proof}[Proof of Lemma~\ref{lem:power:real}]
Suppose that $U,V$ are independent. Then
$W=UV'-U'V\not\equiv0$, and Cauchy's estimate shows that $W$
has finite quadratic type.

\textit{Step 1.}
Apply Lemma~\ref{lem:arithmetic:disk} to $U,W$ with
$\varepsilon=\epsilon$. For all sufficiently large $R$, it gives
a closed disk $\overline{D(\zeta_R,2r)}$ in the square
$R<\operatorname{Re}z,\operatorname{Im}z<2R$, on which $U,W$
do not vanish. By \eqref{eq:power:node-distance} and $h/r\to0$,
we can choose a lattice node $z_0=x_{m_0}+iy_{n_0}$ such that
\[
 |z_0-\zeta_R|\le\sqrt2d_+h<r.
\]
Hence $D(z_0,r)\subset D(\zeta_R,2r)$. On $D(z_0,r)$, define
\[
 Q=\frac{V}{U},\qquad P=Q'=\frac{W}{U^2},\qquad
 A=\frac{U'}U-\frac{W'}{2W}=-\frac{P'}{2P}.
\]
The estimates in Lemma~\ref{lem:arithmetic:disk} give a constant
$K\ge1$ such that, throughout this disk,
\begin{equation}\label{eq:power:A-upper}
 |A|\le KR,\qquad |A'|\le B:=K\log R.
\end{equation}

In Lemma~\ref{lem:arithmetic:lattice}, fix
$h_-=d_-/2$, $h_+=d_++1$ and $M=2$.
Let $T=3h_++3$ and let $c_*>0$ be its lower bound
for $|\lambda|$. Put $c_0=\min\{1,c_*/2\}$.
We claim that
\begin{equation}\label{eq:power:A-lower}
 h|A(z_0)|\ge c_0.
\end{equation}
This is immediate if $h|A(z_0)|>1$. Otherwise, since $Th<r$
for large $R$, define on $D(0,T)$
\[
 U_*(w)=U(z_0+hw),\qquad V_*(w)=V(z_0+hw).
\]
Their Wronskian and logarithmic derivative expression are
\[
 W_*(w)=hW(z_0+hw),\qquad
 A_*(w)=\frac{U_*'(w)}{U_*(w)}-\frac{W_*'(w)}{2W_*(w)}
       =hA(z_0+hw).
\]
In particular, $U_*,W_*$ do not vanish on $D(0,T)$.
Set
\[
 s=\frac{x_{{m_0}+1}-x_{m_0}}{h},\qquad
 t=\frac{y_{{n_0}+1}-y_{n_0}}{h},\qquad \lambda=2hA(z_0).
\]
By \eqref{eq:power:spacing2}, $s,t\in[d_-,d_+]$, and
$|\lambda|\le2$. Also, \eqref{eq:power:A-upper} gives
\[
 \sup_{|w|<T}|2A_*(w)-\lambda|\le2BTh^2.
\]
The node $z_0$, its next two neighbors to the right, and its next
two neighbors above lie in $D(z_0,Th)$.
Their scaled coordinates differ from $0,s,2s,it,2it$ by at most
$J_0h/R$, by \eqref{eq:power:spacing2}, and satisfy
$\operatorname{Im}(V_*\overline{U_*})=0$.
Thus Lemma~\ref{lem:arithmetic:lattice} applies with
$\delta_R=2BTh^2+J_0h/R\to0$ and gives $|\lambda|\ge c_*$.
This proves \eqref{eq:power:A-lower}.

\textit{Step 2.}
We construct a holomorphic logarithm of $Q-c$ on a smaller disk,
where $c$ is real, and estimate its first two derivatives.
Put
\[
 A_0=A(z_0),\qquad M_0=|A_0|,\qquad e=\overline{A_0}/M_0.
\]
Then $|e|=1$ and $A_0e=M_0$. By \eqref{eq:power:A-lower},
\[
 \begin{aligned}
 M_0r&\ge\frac{c_0}{2}R^{\alpha-\epsilon}\longrightarrow\infty,\\
 \frac{Br}{M_0}&\le\frac{K\log R}{2c_0R^{\alpha+\epsilon}}
                  \longrightarrow0,\qquad
 \frac{B}{M_0^2}\le\frac{K\log R}{c_0^2R^{2\alpha}}
                  \longrightarrow0.
 \end{aligned}
\]
For sufficiently large $R$, we have $Br\le M_0/8$.
Integrating the bound on $A'$ in \eqref{eq:power:A-upper}
along segments from $z_0$ gives, on $D(z_0,r)$,
\begin{equation}\label{eq:power:A-comparable}
 |A-A_0|\le M_0/8,\qquad 7M_0/8\le|A|\le9M_0/8.
\end{equation}
By \eqref{eq:power:node-distance}, choose a lattice node $z_*$ with
\[
 |z_*-(z_0+(r/4)e)|\le\sqrt2d_+h\le r/32.
\]
This node lies in $D(z_0,r)$, so $U(z_*)\ne0$ and the sampling
condition gives $Q(z_*)\in\mathbb R$. Set
$c=Q(z_*)$ and $z_1=z_0-(r/4)e$.
For $z\in D(z_1,r/16)$, put $v=z-z_*$. Then
\[
 |v|\le19r/32,\qquad |v+(r/2)e|\le3r/32.
\]
The segment from $z_*$ to $z$ stays in $D(z_0,r)$.
Since $A_0e=M_0$, the second bound gives
\[
\operatorname{Re}(-2A_0v)=M_0r-2\operatorname{Re}\big(A_0(v+\frac{r}{2}e)\big) \ge\frac{13M_0r}{16}.
\]
For $0\le t\le1$, \eqref{eq:power:A-comparable} therefore gives
\begin{align*}
 \operatorname{Re}\big(-2A(z_*+tv)v\big)
 &=\operatorname{Re} \big( -2A_0v \big)-2\operatorname{Re}\Big( \big( A(z_*+tv)-A_0 \big)v \Big) \\
 &\ge\frac{13M_0r}{16}-2\frac{M_0}{8}|v|\\
 &\ge\frac{M_0r}{2}.
\end{align*}
Since $\partial_t \log|P(z_*+tv)|=\operatorname{Re}[-2A(z_*+tv)v]$,
integration from $t$ to $1$ yields
\begin{equation}\label{eq:power:decay}
 \left|\frac{P(z_*+tv)}{P(z)}\right|
 \le e^{-\frac{1}{2}M_0r(1-t)}.
\end{equation}

Integrating $\big(-P/(2A)\big)'=P+PA'/(2A^2)$ along the same segment gives
\begin{align*}
     Q(z)-c &=\int_{z_*}^z \left(-\frac{P}{2A}\right)'-\frac{PA'}{2A^2}\,\rmd\zeta \\
    & =-\frac{P(z)}{2A(z)}+\frac{P(z_*)}{2A(z_*)}
 -\frac12\int_{z_*}^z\frac{P(\zeta)A'(\zeta)}{A(\zeta)^2}
                  \,\rmd\zeta.
\end{align*}
Multiplying by $2A(z)/P(z)$ and rearranging, we obtain
\[
 1+\frac{2A(z)(Q(z)-c)}{P(z)}
 =\frac{A(z)}{A(z_*)}\frac{P(z_*)}{P(z)}
 -\int_{z_*}^z
   \frac{A(z)A'(\zeta)}{A(\zeta)^2}
   \frac{P(\zeta)}{P(z)}\,\rmd\zeta.
\]
By \eqref{eq:power:A-comparable} and $|A'|\le B$,
\[
 \left|\frac{A(z)}{A(z_*)}\right|\le\frac97,
 \qquad
 \left|\frac{A(z)A'(\zeta)}{A(\zeta)^2}\right|
 \le\frac{(9M_0/8)B}{(7M_0/8)^2}
 =\frac{72B}{49M_0}.
\]
Parametrize the segment by $\zeta=z_*+tv$, where $v=z-z_*$
and $0\le t\le1$. Using \eqref{eq:power:decay},
$|v|\le19r/32$, and
\[
 \int_0^1 e^{-M_0r(1-t)/2}\,\rmd t\le\frac{2}{M_0r},
\]
we obtain
\begin{align*}
 \left|1+\frac{2A(z)(Q(z)-c)}{P(z)}\right|
 &\le \frac97e^{-M_0r/2}
   +\frac{72B|v|}{49M_0}
    \int_0^1e^{-M_0r(1-t)/2}\,\rmd t \\
 &\le 2e^{-M_0r/2}+\frac{2B}{M_0^2}.
\end{align*}
By \eqref{eq:power:A-upper} and \eqref{eq:power:A-lower},
$c_0/h\le M_0\le KR$. Thus, for sufficiently large $R$,
\[
 e^{-M_0r/2}
 \le e^{-(c_0/4)R^{\alpha-\epsilon}}
 \le\frac{\log R}{KR^2}
 \le\frac{B}{M_0^2}.
\]
Since $B/M_0^2\to0$, throughout $D(z_1,r/16)$ we have
\begin{equation}\label{eq:power:relative}
 \left|1+\frac{2A(z)(Q(z)-c)}{P(z)}\right|
 \le\frac{4B}{M_0^2}\le\frac14.
\end{equation}
In particular, $Q-c$ has no zero on this disk. Choose a
holomorphic function $H$ there with $e^H=Q-c$.
Since $P'=-2AP$,
\[
 H'=\frac{P}{Q-c},\qquad
 H''=-(H')^2\left(1+\frac{2A(Q-c)}P\right).
\]
Equations \eqref{eq:power:A-comparable} and
\eqref{eq:power:relative} give $7M_0/5\le|H'|\le3M_0$
and $|H''|\le(3M_0)^2(4B/M_0^2)=36B$. Hence
\begin{equation}\label{eq:power:H-bounds}
 \frac{c_0}{h}\le|H'|\le3KR,\qquad |H''|\le36B.
\end{equation}
Define $f(x,y)=\operatorname{Im}H(x+iy)$. This is real harmonic
on $D(z_1,r/16)$ and satisfies
\begin{equation}\label{eq:power:f-bounds}
 |\nabla f|=|H'|,\qquad |f_{xx}|,|f_{xy}|,|f_{yy}|\le36B.
\end{equation}
At every lattice node in this disk, $Q-c$ is real and nonzero, so
\begin{equation}\label{eq:power:integer-phase}
 f(x_m,y_n)\in\pi\mathbb Z.
\end{equation}

\textit{Step 3.}
Let $\rho=r/64$, write $z_1=\xi_1+i\eta_1$, and set
\[
 S_t=\{x+iy:|x-\xi_1|,|y-\eta_1|\le2\rho-t\},
 \qquad 0\le t\le2\rho.
\]
The square $S_0$ lies in $D(z_1,r/16)$.
By \eqref{eq:power:spacing2}, each of its coordinate intervals
contains at least $4\rho/(d_+h)-1$ lattice coordinates.
This number tends to infinity as $R\to\infty$.
For three consecutive horizontal nodes in $S_0$, put
$d_1=x_{m+1}-x_m$ and $d_2=x_{m+2}-x_{m+1}$.
Taylor's formula at $(x_{m+1},y_n)$ and
\eqref{eq:power:f-bounds} give
\begin{align*}
 |f(x_m,y_n)-f(x_{m+1},y_n)+d_1f_x(x_{m+1},y_n)|
 &\le18Bd_1^2,\\
 |f(x_{m+2},y_n)-f(x_{m+1},y_n)-d_2f_x(x_{m+1},y_n)|
 &\le18Bd_2^2.
\end{align*}
The triangle inequality and \eqref{eq:power:spacing2} therefore yield
\[
 \begin{aligned}
 &\left|f(x_{m+2},y_n)-2f(x_{m+1},y_n)+f(x_m,y_n)
       -(d_2-d_1)f_x(x_{m+1},y_n)\right|\\
 &\qquad\le18B(d_1^2+d_2^2)\le36d_+^2Bh^2.
 \end{aligned}
\]
Using \eqref{eq:power:spacing2} and \eqref{eq:power:H-bounds}
in the preceding Taylor estimate, we obtain
\[
 \begin{aligned}
 &|f(x_{m+2},y_n)-2f(x_{m+1},y_n)+f(x_m,y_n)|\\
 &\qquad\le36d_+^2Bh^2+(3KR)\frac{J_0h^2}{R}
 =(36d_+^2B+3KJ_0)h^2.
 \end{aligned}
\]
The same argument gives this bound for the vertical second
difference.

For every lattice cell contained in $S_0$, the fundamental theorem
of calculus gives
\begin{equation}\label{eq:power:mixed-difference}
 \begin{aligned}
 &|f(x_{m+1},y_{n+1})-f(x_{m+1},y_n)
       -f(x_m,y_{n+1})+f(x_m,y_n)|\\
 &\qquad=\left|
   \int_{x_m}^{x_{m+1}}\int_{y_n}^{y_{n+1}}
       f_{xy}(x,y)\,\rmd y\,\rmd x\right|\\
 &\qquad\le36B(x_{m+1}-x_m)(y_{n+1}-y_n)
 \le36d_+^2Bh^2,
 \end{aligned}
\end{equation}
where the inequalities follow from \eqref{eq:power:f-bounds}
and \eqref{eq:power:spacing2}.
Since $Bh^2\to0$, all three differences have modulus less
than $\pi$ for sufficiently large $R$.
They belong to $\pi\mathbb Z$ by \eqref{eq:power:integer-phase},
so they vanish.
Since the horizontal second differences vanish,
$f(x_{m+1},y_n)-f(x_m,y_n)$ is independent of $m$.
The vanishing mixed differences show that this value is
also independent of $n$. The same argument applies vertically.
Since all sampled values belong to $\pi\mathbb Z$, there are
integers $p_R,q_R$ such that
\[
 \begin{aligned}
 f(x_{m+1},y_n)-f(x_m,y_n)&=\pi p_R,\\
 f(x_m,y_{n+1})-f(x_m,y_n)&=\pi q_R
 \end{aligned}
\]
whenever the indicated nodes lie in $S_0$.

Consequently, $f(x_m,y_n)/\pi-p_Rm-q_Rn$ has the same integer
value $k_R$ at every lattice node in $S_0$. Hence
\begin{equation}\label{eq:power:affine}
 f(x_m,y_n)=\pi(p_Rm+q_Rn+k_R).
\end{equation}

\textit{Step 4.}
Put $g(x+iy)=f_{xy}(x,y)=\operatorname{Re}H''(x+iy)$. It is real harmonic on $D(z_1,r/16)$,
and \eqref{eq:power:f-bounds} gives $|g|\le36B$ on $S_0$.
By \eqref{eq:power:mixed-difference} and \eqref{eq:power:affine},
its integral over every lattice cell contained in $S_0$ is zero.
Continuity therefore gives a zero of $g$ in each such cell.

For $\overline{D(w,s)}\subset D(z_1,r/16)$, differentiating
the Poisson formula at the center gives
\begin{equation}\label{eq:power:harmonic-gradient}
 |\nabla g(w)|\le\frac2s\sup_{D(w,s)}|g|.
\end{equation}
Let $\ell=16d_+h$ and $M(t)=\sup_{S_t}|g|$. For large $R$,
$\ell<\rho$. We show that
\begin{equation}\label{eq:M(t+L)}
 M(t+\ell)\le\tfrac12M(t),\qquad 0\le t\le\rho-\ell.
\end{equation}
For $z\in S_{t+\ell}$, a closed lattice cell containing $z$
lies in $S_t$ and contains a zero $\zeta$ of $g$ with
$|z-\zeta|\le\sqrt2d_+h$.
Every point $w$ on the segment from $\zeta$ to $z$ is at least
$\ell-\sqrt2d_+h\ge\ell/2$ from $\partial S_t$.
Applying \eqref{eq:power:harmonic-gradient} with $s=\ell/2$
and integrating along this segment gives
\[
 |g(z)|=|g(z)-g(\zeta)|\le\frac{4\sqrt2d_+h}{\ell}M(t)\le\tfrac12M(t).
\]
Taking the supremum over $z$ yields \eqref{eq:M(t+L)}.
Iterating this estimate $N=\lfloor\rho/\ell\rfloor$ times yields
\[
 \sup_{S_\rho}|g|\le36B\,2^{-N}
 \le72B\exp\left(-\frac{(\log2)\rho}{16d_+h}\right).
\]
For every $z\in S_{3\rho/2}$, the disk $D(z,\rho/2)$ lies
in $S_\rho$.
Applying \eqref{eq:power:harmonic-gradient} on this disk gives
\[
 |\nabla g(z)|
 \le\frac4\rho\sup_{D(z,\rho/2)}|g|
 \le\frac4\rho\sup_{S_\rho}|g|.
\]
Since $f$ is harmonic and $g=f_{xy}$, we have
$f_{xxx}=-g_y$ and $f_{yyy}=-g_x$.
The preceding bound for $\sup_{S_\rho}|g|$ therefore yields,
throughout $S_{3\rho/2}$,
\begin{equation}\label{eq:power:third-derivatives}
 |f_{xxx}|,|f_{yyy}|
 \le|\nabla g|
 \le\frac{288B}{\rho}
       \exp\left(-\frac{(\log2)\rho}{16d_+h}\right).
\end{equation}

\textit{Step 5.}
The coordinate intervals of $S_{3\rho/2}$ have length $\rho$.
By \eqref{eq:power:spacing2} and $\rho/h\to\infty$, for large
$R$ we can choose indices $m_1,n_1$ such that
\[
 x_{m_1+j}+iy_{n_1+k}\in S_{3\rho/2},
 \qquad j,k\in\{0,1,2,3\}.
\]
On $[x_{m_1},x_{m_1+3}]$, define
\[
 \psi(x)=f(x,y_{n_1})
       -\pi\bigl[p_R(x/a)^\beta+q_Rn_1+k_R\bigr].
\]
Equation \eqref{eq:power:affine} gives
$\psi(x_{m_1+j})=0$ for $j=0,1,2,3$.
Three applications of Rolle's theorem give
$\xi\in(x_{m_1},x_{m_1+3})$ such that
\[
 f_{xxx}(\xi,y_{n_1})
 =\pi p_Ra^{-\beta}\beta(\beta-1)(\beta-2)\xi^{\beta-3}.
\]
Set $\gamma_\beta=\pi\beta(\beta-1)(2-\beta)2^{\beta-3}>0$.
Since $\xi^{\beta-3}\ge (2R)^{\beta-3}$, \eqref{eq:power:third-derivatives} gives
\[
 |p_R|\le\frac{288a^\beta}{\gamma_\beta}
 R^{3-\beta}\frac{B}{\rho}
 \exp\left(-\frac{(\log2)\rho}{16d_+h}\right)
 \longrightarrow0.
\]
 Since $p_R$ is an integer, $p_R=0$
for sufficiently large $R$.
Applying the same argument on $[y_{n_1},y_{n_1+3}]$ to
\[
 y\longmapsto f(x_{m_1},y)
       -\pi\bigl[p_Rm_1+q_R(y/b)^\beta+k_R\bigr]
\]
gives $q_R=0$. Thus \eqref{eq:power:affine} makes all the
sampled phase values in $S_0$ equal.

Taylor's formula and
\eqref{eq:power:f-bounds} therefore give
\begin{align*}
 (x_{m_1+1}-x_{m_1})|f_x(x_{m_1},y_{n_1})|
 &\le18B(x_{m_1+1}-x_{m_1})^2,\\
 (y_{n_1+1}-y_{n_1})|f_y(x_{m_1},y_{n_1})|
 &\le18B(y_{n_1+1}-y_{n_1})^2.
\end{align*}
Both increments are positive and at most $d_+h$ by
\eqref{eq:power:spacing2}. Dividing by them yields
\[
 |f_x(x_{m_1},y_{n_1})|,\,
 |f_y(x_{m_1},y_{n_1})|\le18Bd_+h.
\]
Since $|\nabla f|=|H'|$, \eqref{eq:power:H-bounds} now gives
\[
 c_0\le h|\nabla f(x_{m_1},y_{n_1})|
 \le18\sqrt2\,d_+Bh^2\longrightarrow0.
\]
This contradicts $c_0>0$ and proves the lemma.
\end{proof}

\begin{proof}[Proof of Lemma~\ref{lem:power:phase}]
If the first assertion failed, the characterization following
Lemma~\ref{Lm:L5} would give linearly independent entire functions
$U,V$ of finite quadratic type such that
\[
 V(\lambda)\overline{U(\lambda)}\in\mathbb R
 \qquad(\lambda\in\Lambda^+_{a,b,\nu}).
\]
This contradicts Lemma~\ref{lem:power:real}.
The assertion for $F^2(\mathbb C)$ follows from the growth
estimate \eqref{eq:Pf:growth}.
\end{proof}

\begin{proof}[Proof of Theorem~\ref{thm:power}]
Suppose first that $0<\nu<1/2$ and that $f,g\in L^2(\mathbb R)$
have equal sampled Gabor magnitudes. Put $F=P_f$ and $G=P_g$.
By \eqref{eq:vphi:P},
\[
 |F(bm^\nu+ian^\nu)|=|G(bm^\nu+ian^\nu)|,
 \qquad m,n\ge1.
\]
For each $n\ge1$, the entire function
\[
 D_n(w)=F(w+ian^\nu)\overline{F(\bar w+ian^\nu)}
       -G(w+ian^\nu)\overline{G(\bar w+ian^\nu)}
\]
has finite quadratic type by \eqref{eq:ea3} and vanishes at
$bm^\nu$, $m\ge1$. Corollary~\ref{cor:zero-count} gives
$D_n\equiv0$, hence $|F(x+ian^\nu)|=|G(x+ian^\nu)|$
for every $x\in\mathbb R$ and $n\ge1$.
For each fixed $x\in\mathbb R$, the same argument applies to
\[
 E_x(w)=F(x+iw)\overline{F(x+i\bar w)}
       -G(x+iw)\overline{G(x+i\bar w)},
\]
whose zeros include $an^\nu$, $n\ge1$. Thus $E_x\equiv0$,
so $|F|=|G|$ throughout $\mathbb C$, which implies $F\sim G$.
The linearity and injectivity of $f\mapsto P_f$ give $f\sim g$.

The case $\nu=1/2$ is Theorem~\ref{thm:t1}.

If $1/2<\nu<1$, the correspondence established in the proof of
Lemma~\ref{Lm:L5} reduces the assertion to phase retrieval in
$F^2(\mathbb C)$ on
$\Gamma_\Lambda^*=b\mathbb Z^\nu+ia\mathbb Z^\nu$.
This set contains $\Lambda^+_{b,a,\nu}$, which has the required
property by Lemma~\ref{lem:power:phase}.

For $\nu=1$, the nonconstant entire function $Q(z)=e^{\pi z/a}$
has order one and is real at every $z=bm+ian$, $m,n\in\mathbb Z$.
Lemmas~\ref{Lm10}(i) and~\ref{Lm:L5} therefore give nonuniqueness.
For $\nu>1$, counting the nodes in the squares inscribed in
and circumscribed about $D(0,r)$ gives constants $c,C>0$ such that
\[
 cr^{2/\nu}\le n_\Lambda(r)\le Cr^{2/\nu}
\]
for all sufficiently large $r$. Hence $L_\Lambda=2/\nu<2$,
and Lemma~\ref{Th4} gives nonuniqueness.
\end{proof}

\begin{remark}\label{rm:power:first-quadrant}
The first-quadrant sampling set
\[
 \{(am^\nu,bn^\nu):m,n\ge1\}
\]
already does Gabor phase retrieval on $L^2(\mathbb R)$ if
$0<\nu<1$, $\nu\ne1/2$. For $0<\nu<1/2$, the proof of
Theorem~\ref{thm:power} uses only these positive-index samples;
for $1/2<\nu<1$, the assertion follows from
Lemma~\ref{lem:power:phase}.
At $\nu=1/2$, it also suffices if $a^2/b^2$ is irrational,
by Lemmas~\ref{Lm:L5} and~\ref{lem:arithmetic:rigidity}.

At the square-root scale, the strict condition $\min\{a,b\}<1$
also suffices in the first quadrant. For example, if $a<1$,
the positive-ray argument in the proof of Theorem~\ref{thm:t2b}
extends the sampled modulus equality to the three complete
lines through $(0,b)$, $(0,b\sqrt2)$ and $(0,2b)$ in the
time direction. Their distances from the line through $(0,b)$
have ratio $\sqrt2-1$, so
Proposition~\ref{Prop3} applies. The case $b<1$ is analogous.
The rational critical case in Theorem~\ref{thm:t1} is proved
for the full symmetric lattice.
\end{remark}

\section{Gabor phase retrieval for multivariate functions}\label{sec:multivariate}

In this section, we extend Theorem~\ref{thm:power} to coordinatewise power lattices.
Let
\[
 \varphi_d(x)=e^{-\pi|x|^2},\qquad x\in\mathbb R^d,
\]
and define the corresponding Gabor transform by
\[
 V_{\varphi_d}f(t,\omega)=\int_{\mathbb R^d}
 f(u)\varphi_d(u-t)e^{-2\pi i\langle u,\omega\rangle}\,\rmd u.
\]
Here $f_1\sim f_2$ means that $f_1=cf_2$ almost everywhere on
$\mathbb R^d$ for a constant $c$ with $|c|=1$.

\begin{theorem}\label{thm:Gauss:d}
Let $(a_1,\ldots,a_d),(b_1,\ldots,b_d)\in(0,\infty)^d$,
let $\nu_l>0$ for $1\le l\le d$, and put
\[
 \Lambda=\prod_{l=1}^d a_l\mathbb Z^{\nu_l}\times
          \prod_{l=1}^d b_l\mathbb Z^{\nu_l}.
\]
Then $\Lambda$ does Gabor phase retrieval on $L^2(\mathbb R^d)$
if and only if $(a_l,b_l)\in I_{\nu_l}$ for all $1\le l\le d$.
\end{theorem}

\begin{proof}
For $d=1$, this is Theorem~\ref{thm:power}. Assume $d\ge2$ and put
$S_l=a_l\mathbb Z^{\nu_l}\times b_l\mathbb Z^{\nu_l}$.
We regard phase-space coordinates as the ordered pairs
$(t_1,\omega_1),\ldots,(t_d,\omega_d)$.

For an index $l$, a hat denotes deletion of the $l$th coordinate.
Given $(\widehat t_l,\widehat\omega_l)\in\mathbb R^{2(d-1)}$,
define, for almost every $s\in\mathbb R$,
\[
 T_lf(s)=\int_{\mathbb R^{d-1}}
 f(u_1,\ldots,u_{l-1},s,u_{l+1},\ldots,u_d)
 \varphi_{d-1}(\widehat u_l-\widehat t_l)
 e^{-2\pi i\langle\widehat u_l,\widehat\omega_l\rangle}
 \,\rmd\widehat u_l.
\]
Cauchy--Schwarz and Fubini give
\[
 \|T_lf\|_{L^2(\mathbb R)}
 \le\|\varphi_{d-1}\|_{L^2(\mathbb R^{d-1})}
      \|f\|_{L^2(\mathbb R^d)}.
\]
The product formula for the Gaussian and absolute integrability yield
\begin{equation}\label{eq:gaussian:partial-transform}
 V_{\varphi_d}f(t,\omega)
 =V_{\varphi_1}(T_lf)(t_l,\omega_l).
\end{equation}

Suppose first that $(a_l,b_l)\in I_{\nu_l}$ for every $l$,
and that $f_1,f_2$ have equal Gabor magnitudes on $\Lambda$.
We show successively that their magnitudes agree on
\[
 (\mathbb R^2)^k\times\prod_{l=k+1}^d S_l,
 \qquad k=0,\ldots,d.
\]
The case $k=0$ is the hypothesis. Assuming the assertion for
$k-1$, fix the coordinate pairs before the $k$th arbitrarily
in $\mathbb R^2$ and those after it in their sampling sets.
Equation~\eqref{eq:gaussian:partial-transform} gives equal Gabor
magnitudes of $T_kf_1,T_kf_2\in L^2(\mathbb R)$ on $S_k$.
Theorem~\ref{thm:power} yields $T_kf_1\sim T_kf_2$, so their
magnitudes agree for every $(t_k,\omega_k)\in\mathbb R^2$.
This proves the induction, so the original magnitudes agree on
$\mathbb R^{2d}$. For $j=1,2$, define
\[
 H_j(z)=\int_{\mathbb R^d}f_j(u)
 \exp\Big\{-\pi|u|^2+2\pi\sum_{l=1}^d u_lz_l\Big\}\,\rmd u,
 \qquad z\in\mathbb C^d.
\]
Gaussian decay and Cauchy--Schwarz justify differentiation under
the integral on compact subsets, so $H_j$ is entire. Moreover,
\[
 V_{\varphi_d}f_j(t,\omega)=e^{-\pi|t|^2}H_j(t-i\omega).
\]
Thus $|H_1|=|H_2|$ throughout $\mathbb C^d$. If $H_2\not\equiv0$,
the quotient $H_1/H_2$ has constant modulus one on a ball where
$H_2$ does not vanish. The maximum modulus principle and the
identity theorem give $H_1=cH_2$ on $\mathbb C^d$ for some $|c|=1$.
If $H_2\equiv0$, the same identity holds with $c=1$.
The functions $f_j(u)e^{-\pi|u|^2}$ belong to $L^1(\mathbb R^d)$
by Cauchy--Schwarz. Evaluating at $z=-i\omega$ and using Fourier
injectivity therefore gives $f_1(u)e^{-\pi|u|^2}=cf_2(u)e^{-\pi|u|^2}$
almost everywhere. Hence $f_1\sim f_2$.

Conversely, suppose that $(a_{l_0},b_{l_0})\notin I_{\nu_{l_0}}$
for some $l_0$. After permuting coordinates, assume $l_0=1$.
Theorem~\ref{thm:power} supplies $u_1,u_2\in L^2(\mathbb R)$
such that $u_1\not\sim u_2$ and their Gabor magnitudes agree on
$S_1$. Choose $0\ne h\in L^2(\mathbb R^{d-1})$ and set
$F_j(s,\widehat u_1)=u_j(s)h(\widehat u_1)$, $j=1,2$.
The product formula
\[
 V_{\varphi_d}F_j(t,\omega)
 =V_{\varphi_1}u_j(t_1,\omega_1)
  V_{\varphi_{d-1}}h(\widehat t_1,\widehat\omega_1)
\]
shows that their magnitudes agree on $\Lambda$. For every $|c|=1$,
\[
 \|F_1-cF_2\|_2^2
 =\|u_1-cu_2\|_2^2\|h\|_2^2>0.
\]
Thus $F_1\not\sim F_2$, proving necessity.
\end{proof}

\begin{remark}
The positive assertions in Remark~\ref{rm:power:first-quadrant}
also extend coordinatewise: if each triple $(a_l,b_l,\nu_l)$
satisfies one of its sufficient conditions, then
\[
 \prod_{l=1}^d\{a_lm^{\nu_l}:m\ge1\}
 \times\prod_{l=1}^d\{b_ln^{\nu_l}:n\ge1\}
\]
does Gabor phase retrieval on $L^2(\mathbb R^d)$.
Indeed, the preceding proof applies with each $S_l$ restricted
to its first quadrant.
\end{remark}


\begin{thebibliography}{99}

\bibitem{Alaifari2024}
R.~Alaifari, F.~Bartolucci, S.~Steinerberger, and M.~Wellershoff.
\newblock On the connection between uniqueness from samples and stability in
  {G}abor phase retrieval.
\newblock {\em Sampl. Theory Signal Process. Data Anal.}, 22(1):Paper No. 6,
  36, 2024.

\bibitem{AlaifariDaubechiesGrohsYin2019}
R.~Alaifari, I.~Daubechies, P.~Grohs, and R.~Yin.
\newblock Stable phase retrieval in infinite dimensions.
\newblock {\em Found. Comput. Math.}, 19(4):869--900, 2019.

\bibitem{AlaifariGrohs2017}
R.~Alaifari and P.~Grohs.
\newblock Phase retrieval in the general setting of continuous frames for
  {B}anach spaces.
\newblock {\em SIAM J. Math. Anal.}, 49(3):1895--1911, 2017.

\bibitem{AlaifariGrohs2021}
R.~Alaifari and P.~Grohs.
\newblock Gabor phase retrieval is severely ill-posed.
\newblock {\em Appl. Comput. Harmon. Anal.}, 50:401--419, 2021.

\bibitem{AlaifariWellershoff2021b}
R.~Alaifari and M.~Wellershoff.
\newblock Stability estimates for phase retrieval from discrete {G}abor
  measurements.
\newblock {\em J. Fourier Anal. Appl.}, 27(2):Paper No. 6, 31, 2021.

\bibitem{AlaifariWellershoff2021a}
R.~Alaifari and M.~Wellershoff.
\newblock Uniqueness of {STFT} phase retrieval for bandlimited functions.
\newblock {\em Appl. Comput. Harmon. Anal.}, 50:34--48, 2021.

\bibitem{AlaifariWellershoff2022}
R.~Alaifari and M.~Wellershoff.
\newblock Phase retrieval from sampled {G}abor transform magnitudes:
  counterexamples.
\newblock {\em J. Fourier Anal. Appl.}, 28(1):Paper No. 9, 8, 2022.

\bibitem{Boas1954}
R.~P. Boas, Jr.
\newblock {\em Entire functions}.
\newblock Academic Press, Inc., New York, 1954.

\bibitem{CahillCasazzaDaubechies2016}
J.~Cahill, P.~G. Casazza, and I.~Daubechies.
\newblock Phase retrieval in infinite-dimensional {H}ilbert spaces.
\newblock {\em Trans. Amer. Math. Soc. Ser. B}, 3:63--76, 2016.

\bibitem{ChenSun2022}
T.~Chen and W.~Sun.
\newblock Linear phaseless retrieval of functions in spline spaces with
  arbitrary knots.
\newblock {\em IEEE Trans. Inform. Theory}, 68(2):1385--1396, 2022.

\bibitem{ChenChengSun2022}
Y.~Chen, C.~Cheng, and Q.~Sun.
\newblock Phase retrieval of complex and vector-valued functions.
\newblock {\em J. Funct. Anal.}, 283(7):Paper No. 109593, 40, 2022.

\bibitem{ChenChengSunWang2020}
Y.~Chen, C.~Cheng, Q.~Sun, and H.~Wang.
\newblock Phase retrieval of real-valued signals in a shift-invariant space.
\newblock {\em Appl. Comput. Harmon. Anal.}, 49(1):56--73, 2020.

\bibitem{ChengWuXian2025}
C.~Cheng, B.~Wu, and J.~Xian.
\newblock Stable phase retrieval from perturbed samples in shift-invariant
  spaces.
\newblock {\em Appl. Anal.}, 104(14):2715--2738, 2025.

\bibitem{Christensen2016}
O.~Christensen.
\newblock {\em An introduction to frames and {R}iesz bases}.
\newblock Applied and Numerical Harmonic Analysis. Birkh\"auser/Springer,
  [Cham], second edition, 2016.

\bibitem{Daubechies1992}
I.~Daubechies.
\newblock {\em Ten lectures on wavelets}, volume~61 of {\em CBMS-NSF Regional
  Conference Series in Applied Mathematics}.
\newblock Society for Industrial and Applied Mathematics (SIAM), Philadelphia,
  PA, 1992.

\bibitem{DeDiosPontGreilhuberLiehrTaylor2026}
J.~de Dios Pont, J.~Greilhuber, L.~Liehr, and M.~A. Taylor.
\newblock Phase retrieval from a uniformly discrete point set.
\newblock arXiv:2609.28341, 2026.

\bibitem{EvansGariepy2015}
L.~C. Evans and R.~F. Gariepy.
\newblock {\em Measure theory and fine properties of functions}.
\newblock CRC Press, Boca Raton, FL, revised edition, 2015.

\bibitem{FreemanOikhbergPineauTaylor2024}
D.~Freeman, T.~Oikhberg, B.~Pineau, and M.~A. Taylor.
\newblock Stable phase retrieval in function spaces.
\newblock {\em Math. Ann.}, 390(1):1--43, 2024.

\bibitem{Gao2025}
B.~Gao.
\newblock Affine phase retrieval via second-order methods.
\newblock {\em Inverse Problems}, 41(5):Paper No. 055011, 28, 2025.

\bibitem{GaoSunWangXu2018}
B.~Gao, Q.~Sun, Y.~Wang, and Z.~Xu.
\newblock Phase retrieval from the magnitudes of affine linear measurements.
\newblock {\em Adv. in Appl. Math.}, 93:121--141, 2018.

\bibitem{Groechenig2001}
K.~Gr\"ochenig.
\newblock {\em Foundations of time-frequency analysis}.
\newblock Applied and Numerical Harmonic Analysis. Birkh\"auser Boston, Inc.,
  Boston, MA, 2001.

\bibitem{Grochenig2020}
K.~Gr\"ochenig.
\newblock Phase-retrieval in shift-invariant spaces with {G}aussian generator.
\newblock {\em J. Fourier Anal. Appl.}, 26(3):Paper No. 52, 15, 2020.

\bibitem{GrohsKoppensteinerRathmair2020}
P.~Grohs, S.~Koppensteiner, and M.~Rathmair.
\newblock Phase retrieval: uniqueness and stability.
\newblock {\em SIAM Rev.}, 62(2):301--350, 2020.

\bibitem{GrohsLiehr2023b}
P.~Grohs and L.~Liehr.
\newblock Injectivity of {G}abor phase retrieval from lattice measurements.
\newblock {\em Appl. Comput. Harmon. Anal.}, 62:173--193, 2023.

\bibitem{GrohsLiehr2023a}
P.~Grohs and L.~Liehr.
\newblock Non-uniqueness theory in sampled {STFT} phase retrieval.
\newblock {\em SIAM J. Math. Anal.}, 55(5):4695--4726, 2023.

\bibitem{GrohsLiehr2024}
P.~Grohs and L.~Liehr.
\newblock Stable {G}abor phase retrieval in {G}aussian shift-invariant spaces
  via biorthogonality.
\newblock {\em Constr. Approx.}, 59(1):61--111, 2024.

\bibitem{GrohsLiehr2025}
P.~Grohs and L.~Liehr.
\newblock Phaseless sampling on square-root lattices.
\newblock {\em Found. Comput. Math.}, 25(2):351--374, 2025.

\bibitem{GrohsLiehrRathmair2025b}
P.~Grohs, L.~Liehr, and M.~Rathmair.
\newblock Multi-window {STFT} phase retrieval: lattice uniqueness.
\newblock {\em J. Funct. Anal.}, 288(3):Paper No. 110733, 23, 2025.

\bibitem{GrohsLiehrRathmair2025a}
P.~Grohs, L.~Liehr, and M.~Rathmair.
\newblock Phase retrieval in {F}ock space and perturbation of {L}iouville sets.
\newblock {\em Rev. Mat. Iberoam.}, 41(3):969--1008, 2025.

\bibitem{Grohs2019}
P.~Grohs and M.~Rathmair.
\newblock Stable {G}abor phase retrieval and spectral clustering.
\newblock {\em Comm. Pure Appl. Math.}, 72(5):981--1043, 2019.

\bibitem{Grohs2022}
P.~Grohs and M.~Rathmair.
\newblock Stable {G}abor phase retrieval for multivariate functions.
\newblock {\em J. Eur. Math. Soc. (JEMS)}, 24(5):1593--1615, 2022.

\bibitem{Hormander1990}
L.~H\"ormander.
\newblock {\em An introduction to complex analysis in several variables},
  volume~7 of {\em North-Holland Mathematical Library}.
\newblock North-Holland Publishing Co., Amsterdam, third edition, 1990.

\bibitem{Hormander1994}
L.~H\"ormander.
\newblock {\em Notions of convexity}, volume~127 of
  {\em Progress in Mathematics}.
\newblock Birkh\"auser Boston, Inc., Boston, MA, 1994.

\bibitem{HuangLiXu2025}
G.~Huang, S.~Li, and H.~Xu.
\newblock Adversarial phase retrieval via nonlinear least absolute deviation.
\newblock {\em IEEE Trans. Inform. Theory}, 71(9):7396--7415, 2025.

\bibitem{HuangLiXu2026}
G.~Huang, S.~Li, and H.~Xu.
\newblock Robust outlier bound condition to phase retrieval with adversarial
  sparse outliers.
\newblock {\em Appl. Comput. Harmon. Anal.}, 80:Paper No. 101819, 2026.

\bibitem{HuangXu2024}
M.~Huang and Z.~Xu.
\newblock Strong convexity of affine phase retrieval.
\newblock {\em IEEE Trans. Signal Process.}, 72:1301--1315, 2024.

\bibitem{Jaming2014}
P.~Jaming.
\newblock Uniqueness results in an extension of {P}auli's phase retrieval
  problem.
\newblock {\em Appl. Comput. Harmon. Anal.}, 37(3):413--441, 2014.

\bibitem{Levin1980}
B.~J. Levin.
\newblock {\em Distribution of zeros of entire functions}, volume~5 of {\em
  Translations of Mathematical Monographs}.
\newblock American Mathematical Society, Providence, RI, revised edition, 1980.
\newblock Translated from the Russian by R. P. Boas, J. M. Danskin, F. M.
  Goodspeed, J. Korevaar, A. L. Shields and H. P. Thielman.

\bibitem{Levin1996}
B.~Ya. Levin.
\newblock {\em Lectures on entire functions}, volume~150 of
  {\em Translations of Mathematical Monographs}.
\newblock American Mathematical Society, Providence, RI, 1996.
\newblock In collaboration with Yu. Lyubarskii, M. Sodin, and
  V. Tkachenko. Translated by Vadim Tkachenko from an original
  Russian manuscript.

\bibitem{LiLi2021b}
H.~Li and S.~Li.
\newblock Phase retrieval from {F}ourier measurements with masks.
\newblock {\em Inverse Probl. Imaging}, 15(5):1051--1075, 2021.

\bibitem{LiLi2021a}
H.~Li and S.~Li.
\newblock Riemannian optimization for phase retrieval from masked {F}ourier
  measurements.
\newblock {\em Adv. Comput. Math.}, 47(6):Paper No. 88, 32, 2021.

\bibitem{LiLiuZhang2021}
R.~Li, B.~Liu, and Q.~Zhang.
\newblock Uniqueness of {STFT} phase retrieval in shift-invariant spaces.
\newblock {\em Appl. Math. Lett.}, 118:Paper No. 107131, 6, 2021.

\bibitem{LiWeiFan2024}
Y.~Li, X.~Wei, and S.~Fan.
\newblock Single-shot phase retrieval by interference intensity: a
  holography-driven problem for periodic signals.
\newblock {\em IEEE Trans. Inform. Theory}, 70(5):3767--3787, 2024.

\bibitem{LiYang2025}
Y.-Z. Li and M.~Yang.
\newblock Affine phase retrieval of quaternion signals.
\newblock {\em Linear Multilinear Algebra}, 73(5):865--882, 2025.

\bibitem{Liehr2024}
L.~Liehr.
\newblock Arithmetic progressions and holomorphic phase retrieval.
\newblock {\em Bull. Lond. Math. Soc.}, 56(11):3316--3330, 2024.
\newblock doi:10.1112/blms.13134.

\bibitem{Pavlovic2004}
M.~Pavlovi\'c.
\newblock {\em Introduction to function spaces on the disk},
  volume~20 of {\em Posebna Izdanja}.
\newblock Matemati\v{c}ki Institut SANU, Beograd, 2004.

\bibitem{QianTan2016}
T.~Qian and L.~Tan.
\newblock Backward shift invariant subspaces with applications to band
  preserving and phase retrieval problems.
\newblock {\em Math. Methods Appl. Sci.}, 39(6):1591--1598, 2016.

\bibitem{Ramachandra1995}
K.~Ramachandra.
\newblock {\em On the mean-value and omega-theorems for the {R}iemann
  zeta-function}, volume~85 of {\em Tata Institute of Fundamental
  Research Lectures on Mathematics and Physics}.
\newblock Springer-Verlag, Berlin, 1995.

\bibitem{Rudin1987}
W.~Rudin.
\newblock {\em Real and complex analysis}.
\newblock McGraw-Hill Book Co., New York, third edition, 1987.

\bibitem{Shechtman2015}
Y.~Shechtman, Y.~C. Eldar, O.~Cohen, H.~N. Chapman, J.~Miao, and M.~Segev.
\newblock Phase retrieval with application to optical imaging: A contemporary
  overview.
\newblock {\em IEEE Signal Processing Magazine}, 32(3):87--109, May 2015.

\bibitem{ShenoyMulletiSeelamantula2016}
B.~A. Shenoy, S.~Mulleti, and C.~S. Seelamantula.
\newblock Exact phase retrieval in principal shift-invariant spaces.
\newblock {\em IEEE Trans. Signal Process.}, 64(2):406--416, 2016.

\bibitem{Sun2021}
W.~Sun.
\newblock Local and global phaseless sampling in real spline spaces.
\newblock {\em Math. Comp.}, 90(330):1899--1929, 2021.

\bibitem{Wellershoff2023}
M.~Wellershoff.
\newblock Sampling at twice the {N}yquist rate in two frequency bins guarantees
  uniqueness in {G}abor phase retrieval.
\newblock {\em J. Fourier Anal. Appl.}, 29(1):Paper No. 7, 9, 2023.

\bibitem{Wellershoff2024b}
M.~Wellershoff.
\newblock Injectivity of sampled {G}abor phase retrieval in spaces with general
  integrability conditions.
\newblock {\em J. Math. Anal. Appl.}, 530(2):Paper No. 127692, 10, 2024.

\bibitem{Wellershoff2024a}
M.~Wellershoff.
\newblock Phase retrieval of entire functions and its implications for {G}abor
  phase retrieval.
\newblock {\em J. Funct. Anal.}, 286(11):Paper No. 110403, 29, 2024.

\bibitem{XiaXu2026}
Y.~Xia and Z.~Xu.
\newblock Instance optimality in phase retrieval.
\newblock {\em Appl. Comput. Harmon. Anal.}, 80:Paper No. 101818, 2026.

\bibitem{XiaXuXu2025}
Y.~Xia, Z.~Xu, and Z.~Xu.
\newblock Stability in phase retrieval: characterizing condition numbers and
  the optimal vector set.
\newblock {\em Math. Comp.}, 94(356):2931--2960, 2025.

\bibitem{YangLi2024}
M.~Yang and Y.-Z. Li.
\newblock Generalized phase retrieval in quaternion {E}uclidean spaces.
\newblock {\em Math. Methods Appl. Sci.}, 47(18):14699--14717, 2024.

\bibitem{ZhangGuoLiuLi2023}
Q.~Zhang, Z.~Guo, B.~Liu, and R.~Li.
\newblock Uniqueness of {STFT} phase retrieval for bandlimited vector
  functions.
\newblock {\em Numer. Funct. Anal. Optim.}, 44(4):311--331, 2023.

\bibitem{Zhong2023}
S.~Zhong, B.~Liu, R.~Li, and Q.~Zhang.
\newblock Stability estimates for phase retrieval from discrete linear
  canonical {G}abor transformation measurements.
\newblock {\em Math. Methods Appl. Sci.}, 46(4):3937--3947, 2023.

\end{thebibliography}
\end{document}